\documentclass{amsart}

\usepackage{mathtools,amssymb,amsthm}
\usepackage{enumitem}

\theoremstyle{plain}
\newtheorem{theorem}{Theorem}[section]
\newtheorem{lemma}[theorem]{Lemma}
\newtheorem{proposition}[theorem]{Proposition}
\newtheorem{corollary}[theorem]{Corollary}
\newtheorem{conjecture}[theorem]{Conjecture}
\theoremstyle{remark}
\newtheorem{remark}[theorem]{Remark}

\newcommand{\CC}{\mathbb{C}}
\newcommand{\FF}{\mathbb{F}}

\newcommand{\QQ}{\mathbb{Q}}

\newcommand{\ZZ}{\mathbb{Z}}

\DeclareMathOperator{\Cl}{Cl}
\DeclareMathOperator{\Gal}{Gal}
\DeclareMathOperator{\Norm}{N}
\DeclareMathOperator{\id}{id}
\DeclareMathOperator{\lcm}{lcm}
\DeclareMathOperator{\rad}{rad}

\newcommand{\Frob}[2]{\left[\dfrac{#1}{#2}\right]}

\begin{document}

\title[Class number divisibility for pairs of real quadratic fields]{Infinite families of pairs of real quadratic fields \\
whose class numbers are divisible \\ by a given integer}
\author[Y.~Iizuka]{Yoshichika Iizuka}
\address{Department of Mathematics, Gakushuin University, Mejiro, Toshima-ku, Tokyo 171-8588, Japan}
\email{iizuka@math.gakushuin.ac.jp}
\date{}


\begin{abstract}
For any integers \(N \ge 2\) and \(m \ge 1\), we prove that there exist infinitely many pairs of real quadratic fields \(\QQ(\sqrt{D}), \QQ(\sqrt{D + m})\), with \(D \in \ZZ\) and \(D > 0\), such that the class numbers of both fields are divisible by \(N\). Write \(N = 2^{e}n\) with \(n\) odd. If \(n > 1\), then the class groups of both fields in each such pair contain an element of order \(n\).
\end{abstract}

\maketitle

\section{Introduction}

The divisibility of class numbers of quadratic fields is a classical topic in number theory. In particular, a basic question is whether there exist infinitely many quadratic fields whose class numbers are divisible by a given integer. It is known from Gauss's genus theory that the class number of a quadratic field is divisible by a power of \(2\) depending on the number of distinct prime factors of the discriminant of the field. For imaginary quadratic fields, Nagell proved in 1922 that, for any positive integer \(n\), there exist infinitely many imaginary quadratic fields whose class numbers are divisible by \(n\) \cite{Nagell1922}. Subsequently, Ankeny and Chowla gave another proof in 1955 \cite{AnkenyChowla1955}.

For real quadratic fields, Honda's pioneering result \cite{Honda1968} establishes the existence of infinitely many real quadratic fields whose class numbers are divisible by \(3\). Another classical result concerning the \(3\)-divisibility of class numbers is Scholz's reflection theorem, which relates the \(3\)-ranks of the class groups of real and imaginary quadratic fields \cite{Scholz1932}. According to Scholz's result, if \(d > 1\) is a square-free integer, \(r\) is the \(3\)-rank of the class group of the real quadratic field \(\QQ(\sqrt{d})\), and \(s\) is the \(3\)-rank of the class group of the corresponding imaginary quadratic field \(\QQ(\sqrt{-3d})\), then we have
\[
r \le s \le r + 1.
\]
For the definition of the \(p\)-rank of a finite abelian group for a prime \(p\), see Section~\ref{sec:notation}. Consequently, if the class number of \(\QQ(\sqrt{d})\) is a multiple of \(3\), then the class number of \(\QQ(\sqrt{-3d})\) is also a multiple of \(3\). This theorem illustrates the close relationship between the divisibility properties of the class numbers of real and imaginary quadratic fields.

In 1970, Yamamoto proved that, for any positive integer \(n\), there exist infinitely many real quadratic fields whose class numbers are multiples of \(n\) \cite{Yamamoto1970}. Yamamoto's method, which uses ideas from class field theory, relates the construction of unramified extensions to integer solutions of Diophantine equations. More precisely, two ideal classes are constructed from integer solutions of a certain Diophantine equation; in the imaginary quadratic case, both of them are shown to have order \(n\) and to be independent, whereas in the real quadratic case, the subgroup generated by them is shown to contain an ideal class of order \(n\). Furthermore, Yamamoto deduces the existence of infinitely many such quadratic fields by imposing suitable local conditions on the discriminant. The construction in the real quadratic case is more complicated than in the imaginary quadratic case because real quadratic fields have units of infinite order. Subsequently, in 1973, Weinberger proved that, for any positive integer \(n\), there exist infinitely many real quadratic fields of the form \(\QQ(\sqrt{x^{2n} + 4})\) (\(x \in \ZZ\)) whose class numbers are multiples of \(n\) \cite{Weinberger1973}.

Subsequent work turned to pairs of quadratic fields and to the problem of simultaneous divisibility of their class numbers by a specified integer. In particular, concerning the \(3\)-divisibility of class numbers, Komatsu gave an infinite family of pairs of quadratic fields of the form \(\QQ(\sqrt{D})\) and \(\QQ(\sqrt{mD})\) (\(D \in \ZZ\)) such that both class numbers are multiples of \(3\) \cite{Komatsu2002}. Komatsu later extended this result to arbitrary orders: for any natural numbers \(n, m > 1\), he constructed an infinite family of pairs of imaginary quadratic fields \(\QQ(\sqrt{D})\) and \(\QQ(\sqrt{mD})\) such that the class group of each field contains an element of order \(n\) \cite{Komatsu2017}. In addition, Iizuka, Konomi, and Nakano constructed, for any fixed \(l \in \{3, 5, 7\}\) and integers \(a, b\) (\(a \ne 0\)), an infinite family of pairs of quadratic fields of the form \(\QQ(\sqrt{D})\) and \(\QQ(\sqrt{aD + b})\) (\(D \in \QQ\)) such that both class numbers are multiples of \(l\) \cite{IizukaKonomiNakano2016, IizukaKonomiNakano2021}. Building on these works, the author focused on pairs of imaginary quadratic fields \(\QQ(\sqrt{D}), \QQ(\sqrt{D + 1})\) (\(D \in \ZZ\)) and constructed an infinite family of such pairs whose class numbers are both multiples of \(3\) \cite{Iizuka2018}. A key feature of this result is that the radicand \(D\) is taken to be an integer and that the construction gives an infinite family in which the class numbers of the imaginary quadratic fields obtained from two radicands differing by \(1\) are both multiples of \(3\). The author further proposed the following conjecture as a generalization:

\begin{conjecture}\label{conj:iizuka}
For any prime \(l\) and any positive integer \(m\),
there exist infinitely many distinct \((m + 1)\)-tuples of
real (resp.~imaginary) quadratic fields
\[
\left(\QQ(\sqrt{D}), \QQ(\sqrt{D + 1}), \dots, \QQ(\sqrt{D + m})\right),
\qquad D \in \ZZ,
\]
such that the class number of each field is divisible by \(l\).
\end{conjecture}

In a related direction, Chattopadhyay and Muthukrishnan extended the simultaneous \(3\)-divisibility problem from pairs to certain triples. More precisely, if \(k \ge 1\) is cube-free, \(k \equiv 1 \pmod{9}\), and \(\gcd(k, 7 \cdot 571) = 1\), then they constructed infinitely many triples of imaginary quadratic fields
\[
\left(\QQ(\sqrt{D}), \QQ(\sqrt{D + 1}), \QQ(\sqrt{D + k^{2}})\right),
\qquad D \in \ZZ,
\]
whose class numbers are all divisible by \(3\) \cite{ChattopadhyayMuthukrishnan2021}.

The case \(l = 2\) of Conjecture~\ref{conj:iizuka} can be proved by genus theory and the Chinese remainder theorem \cite[\S3]{Iizuka2018}.
A proof that makes explicit the distinction between the class number and the narrow class number for real quadratic fields, as well as the existence of infinitely many distinct tuples of fields, is given in Appendix~\ref{app:case-l-two}.
Moreover, when \(m = 1\), Conjecture~\ref{conj:iizuka} is settled for all primes \(l\):
in the imaginary quadratic case by the result of Xie and Chao \cite{XieChao2020} described below (together with \cite[\S3]{Iizuka2018} for \(l = 2\)),
and in the real quadratic case by the main result of the present paper (Corollary~\ref{cor:conjecture-real-pairs}).

Xie and Chao substantially extended the author's result: for any odd positive integer \(n\) and any positive integer \(m\), they constructed an infinite family of pairs of imaginary quadratic fields
\[
\left(\QQ(\sqrt{D}), \QQ(\sqrt{D + m})\right),
\qquad D \in \ZZ, \quad D + m < 0,
\]
such that the class group of each field contains an element of order \(n\) \cite{XieChao2020}. Consequently, besides the case \(n = 3\), one obtains infinitely many examples for odd primes \(n\), such as \(5\) and \(7\), in which the class numbers of both fields are divisible by \(n\). The result of Xie and Chao \cite{XieChao2020} generalizes the author's result for \(n = 3\) and \(m = 1\). In their construction, auxiliary primes are chosen by means of a local--global principle for several binomials \(x^{n} - a\).
Furthermore, by imposing congruence conditions at each of these primes on the parameter that yields integer solutions of a Diophantine equation, they arrange that the two quadratic fields have the required property simultaneously.
To prove the divisibility of the class numbers, they use Yamamoto's idea \cite{Yamamoto1970}.

The case of pairs of real quadratic fields is more difficult because, unlike imaginary quadratic fields, real quadratic fields have units of infinite order. Xie and Chao constructed, for any positive integer \(m\), an infinite family of pairs of real quadratic fields \(\QQ(\sqrt{D}), \QQ(\sqrt{D + m})\) whose class numbers are both multiples of \(3\) \cite{XieChao2022}. Their construction of the pairs follows the method developed in \cite{XieChao2020}, whereas their proof of the \(3\)-divisibility of the class numbers uses the criterion of Kishi and Miyake \cite{KishiMiyake2000}.

The purpose of the present paper is to extend the above result of Xie and Chao \cite{XieChao2022} to divisibility by an arbitrary integer. The main result is as follows:

\begin{theorem}\label{thm:main}
For any integer \(N \ge 2\) and any integer \(m \ge 1\),
there exist infinitely many distinct pairs of real quadratic fields
\[
\left(\QQ(\sqrt{D}), \QQ(\sqrt{D + m})\right),
\qquad D \in \ZZ, \quad D > 0,
\]
such that both class numbers are divisible by \(N\).
Moreover, if \(N = 2^{e}n\) with \(n\) odd and \(n > 1\), then
the class group of each field contains an element of order \(n\). In particular, if \(N\) is odd, then the class group of each field contains an element of order \(N\).
\end{theorem}

The stronger assertion for odd \(N\), namely that the class group of each field contains an element of order \(N\),
is a real quadratic analogue of the result of Xie and Chao for pairs of imaginary quadratic fields \cite[Theorem~1.2]{XieChao2020}.

Applying Theorem~\ref{thm:main} with \(m = 1\) and \(N = l\), where \(l\) is a prime,
we obtain the case \(m = 1\) of Conjecture~\ref{conj:iizuka} for real quadratic fields, for every prime \(l\).

\begin{corollary}\label{cor:conjecture-real-pairs}
For any prime \(l\), there exist infinitely many distinct pairs of real quadratic fields
\[
\left(\QQ(\sqrt{D}), \QQ(\sqrt{D + 1})\right),
\qquad D \in \ZZ, \quad D > 0,
\]
such that both class numbers are divisible by \(l\).
In other words, the case \(m = 1\) of Conjecture~\ref{conj:iizuka} holds for real quadratic fields for every prime \(l\).
Moreover, the class group of each field contains an element of order \(l\).
\end{corollary}

\begin{proof}
If we take \(m = 1\) and \(N = l\) in Theorem~\ref{thm:main}, then both class numbers are divisible by \(l\).
In addition, Cauchy's theorem for finite groups implies that the class group of each field contains an element of order \(l\).
\end{proof}

For imaginary quadratic fields, the case \(m = 1\) is already known: Xie and Chao proved it for odd primes \(l\) \cite[Theorem~1.2]{XieChao2020},
while the case \(l = 2\) follows from \cite[\S3]{Iizuka2018}; see also Appendix~\ref{app:case-l-two}.
Combined with Corollary~\ref{cor:conjecture-real-pairs},
this shows that the case \(m = 1\) of Conjecture~\ref{conj:iizuka} is completely settled for both real and imaginary quadratic fields.
On the other hand, for \(m \ge 2\), Conjecture~\ref{conj:iizuka} requires \((m + 1)\)-tuples,
whereas Theorem~\ref{thm:main} produces only pairs whose radicands differ by \(m\); hence the conjecture remains open for general \(m\).

We prove the theorem by writing \(N = 2^{e}n\), where \(e \ge 0\) and \(n\) is odd.
When the odd part \(n\) is greater than \(1\), we control the \(p\)-th power residuosity of the fundamental unit for each prime divisor \(p\) of \(n\)
using an explicitly constructed unit and auxiliary primes,
and we show that the class group of each field contains an element of order \(n\).
In choosing the auxiliary primes, we do not rely on the general statement of the local--global principle used by Xie and Chao \cite{XieChao2020};
instead, we establish the required linear independence in Lemma~\ref{lem:AE-independent},
and then prove the existence of the required auxiliary primes directly using Chebotarev's density theorem in Lemma~\ref{lem:choice-of-auxiliary-primes}.

We treat the \(2\)-part separately, using Gauss's genus theory rather than Kummer theory.
We arrange in advance that the discriminant of each real quadratic field has at least \(e + 2\) distinct prime divisors.
By genus theory, the \(2\)-rank of the narrow class group of a quadratic field is
one less than the number of distinct prime factors of the discriminant, so that under this ramification condition the narrow class number is divisible by \(2^{e + 1}\).
Since the narrow class number of a real quadratic field is equal to either the ordinary class number or twice the ordinary class number,
it follows that \(2^{e}\) divides the ordinary class number (Lemma~\ref{lem:genus-two-part}).
This ramification condition is realized by choosing additional primes \(\xi_{j, k}\) in Lemma~\ref{lem:choice-of-H},
and it is satisfied simultaneously with the system of congruence conditions~\eqref{eq:congruence-system} on the final parameter \(t\).
Thus the residue conditions for the odd part, which come from Kummer extensions and Chebotarev's density theorem, and the genus-theoretic conditions for the \(2\)-part are prepared independently of each other, and they are finally combined into a single family by the Chinese remainder theorem.
When the odd part is \(1\), that is, when \(N\) is a power of \(2\), we treat the case directly, using only genus theory and the Chinese remainder theorem.

Furthermore, by adapting the construction in this paper to the imaginary quadratic case, we obtain another proof of the result of Xie and Chao \cite[Theorem~1.2]{XieChao2020}.
This supplementary application is presented in Section~\ref{sec:imaginary-quadratic-case}.
For an imaginary quadratic field other than the two exceptional fields \(\QQ(\sqrt{-1})\) and \(\QQ(\sqrt{-3})\), the unit group is \(\{\pm 1\}\), so that the control of the \(p\)-th power residuosity of the fundamental unit required in the real quadratic case is unnecessary.

\section{Notation}\label{sec:notation}

In this section, we collect the notation and terminology used throughout the paper.

\subsection{Numbers, rings, and fields}

For a finite set \(S\), we denote its cardinality by \(|S|\).

We denote by \(\ZZ\), \(\QQ\), and \(\CC\) the ring of rational integers, the field of rational numbers, and the field of complex numbers, respectively, and by \(\overline{\QQ}\) the algebraic closure of \(\QQ\) in \(\CC\). All number fields are regarded as subfields of \(\overline{\QQ}\). For a prime \(q\), let \(\FF_{q}\) denote the field with \(q\) elements.

For a number field \(K\), let \(\mathcal{O}_{K}\) denote the ring of integers of \(K\).

For a commutative ring \(R\), we denote by \(R^{\times}\) the unit group of \(R\), that is, the multiplicative group of all invertible elements. In particular, for a field \(K\), the set \(K^{\times} = K \setminus \{0\}\) is the multiplicative group of \(K\), and for a number field \(K\), the group \(\mathcal{O}_{K}^{\times}\) is the unit group of \(K\).

\subsection{Notation for integers}

For a nonzero integer \(a\), let \(\omega(a)\) denote the number of distinct prime factors of \(a\), and let \(\rad(a)\) denote the product of the distinct prime factors of \(a\). The greatest common divisor and the least common multiple of integers are denoted by \(\gcd\) and \(\lcm\), respectively.

A nonzero integer \(D\) can be written uniquely as \(D = du^{2}\) with a square-free integer \(d\) (which carries the sign of \(D\)) and a positive integer \(u\). We call this \(d\) the square-free part of \(D\). In particular, when \(D\) is not a square, the discriminant of the quadratic field \(\QQ(\sqrt{D})\) is \(d\) or \(4d\).

For a prime \(p\), we denote by \(v_{p}\colon \QQ^{\times} \to \ZZ\) the \(p\)-adic valuation; that is, for \(\alpha \in \QQ^{\times}\), the value \(v_{p}(\alpha)\) is the exponent of \(p\) in the prime factorization of \(\alpha\).

\subsection{Notation for groups}

For an abelian group \(G\) written multiplicatively and a positive integer \(r\), we put
\[
G^{r} = \{g^{r} : g \in G\}.
\]
When this notation is applied to multiplicative groups and unit groups, we omit the parentheses and write
\[
K^{\times r} = (K^{\times})^{r}, \qquad
\FF_{q}^{\times r} = (\FF_{q}^{\times})^{r}, \qquad
(\mathcal{O}_{K}/\mathfrak{q})^{\times r} = \bigl((\mathcal{O}_{K}/\mathfrak{q})^{\times}\bigr)^{r}.
\]
In the first expression, \(K\) is a field, and in the second expression, \(q\) is a prime.
In the last expression, \(K\) is a number field and \(\mathfrak{q}\) is a prime ideal of \(K\).

For an abelian group \(A\) written additively and a positive integer \(r\), we denote by \(A^{\oplus r}\) the direct sum of \(r\) copies of \(A\):
\[
A^{\oplus r} = \{(u_{1}, \dots, u_{r}) : u_{1}, \dots, u_{r} \in A\}.
\]
In particular, \(A^{\oplus 2} = A \oplus A\).
The direct product of groups written multiplicatively is denoted by \(\times\); to avoid confusion with the notation \(G^{r}\) above, direct sums are always written with \(\oplus\).
In this paper, we mainly use \(\FF_{p}^{\oplus r}\) for a prime \(p\), that is, the standard \(r\)-dimensional vector space over \(\FF_{p}\).

For an element \(g\) of a group \(G\), we denote by \(\langle g \rangle\) the cyclic subgroup generated by \(g\).

For a finite abelian group \(A\) written multiplicatively and a prime \(p\), the \(p\)-rank of \(A\) is defined to be \(\dim_{\FF_{p}}(A/A^{p})\).
If the \(p\)-rank of \(A\) is \(r\), then \(|A/A^{p}| = p^{r}\), and in particular \(p^{r}\) divides \(|A|\).

\subsection{Field extensions and automorphisms}

For a field \(L\), let \(\id_{L}\) denote the identity map of \(L\). For a Galois extension \(L/K\) of fields, we denote its Galois group by \(\Gal(L/K)\). The identity element of \(\Gal(L/K)\) is \(\id_{L}\).

We write the action of a field automorphism \(\sigma\) on an element \(x\) or an ideal \(\mathfrak{a}\) as \(\sigma(x)\), \(\sigma(\mathfrak{a})\), or equivalently as \(x^{\sigma}\), \(\mathfrak{a}^{\sigma}\).

For a finite extension \(K/F\) of fields, we denote by \([K:F]\) its degree and by \(\Norm_{K/F}\colon K^{\times} \to F^{\times}\) the norm map.

\subsection{Prime ideals and valuations}

Throughout the paper, a prime ideal of a number field always means a nonzero prime ideal.

For a prime ideal \(\mathfrak{p}\) of a number field \(K\), we denote by \(v_{\mathfrak{p}}\colon K^{\times} \to \ZZ\) the \(\mathfrak{p}\)-adic valuation; that is, for \(\alpha \in K^{\times}\), the value \(v_{\mathfrak{p}}(\alpha)\) is the exponent of \(\mathfrak{p}\) in the prime ideal factorization of the principal ideal \((\alpha)\).

For a prime ideal \(\mathfrak{q}\) of a number field \(K\) and \(\alpha \in \mathcal{O}_{K}\), we say that \(\mathfrak{q}\) divides \(\alpha\) and write \(\mathfrak{q} \mid \alpha\) if \(\alpha \in \mathfrak{q}\); otherwise we write \(\mathfrak{q} \nmid \alpha\). When \(\alpha \ne 0\), the condition \(\mathfrak{q} \nmid \alpha\) is equivalent to \(v_{\mathfrak{q}}(\alpha) = 0\). For a rational prime \(p\), we similarly say that \(\mathfrak{q}\) divides \(p\) if \(p \in \mathfrak{q}\); equivalently, \(\mathfrak{q}\) lies above \(p\).

For nonzero ideals \(\mathfrak{a}\), \(\mathfrak{b}\) of a number field \(K\), we say that \(\mathfrak{a}\) divides \(\mathfrak{b}\) and write \(\mathfrak{a} \mid \mathfrak{b}\) if \(\mathfrak{b} \subseteq \mathfrak{a}\); otherwise we write \(\mathfrak{a} \nmid \mathfrak{b}\). For a prime ideal \(\mathfrak{q}\) and \(\alpha \in \mathcal{O}_{K} \setminus \{0\}\), the condition \(\mathfrak{q} \mid (\alpha)\) is equivalent to the condition \(\mathfrak{q} \mid \alpha\) above.

For a prime ideal \(\mathfrak{q}\) of a number field \(K\) and \(\alpha, \beta \in \mathcal{O}_{K}\), we write \(\alpha \equiv \beta \pmod{\mathfrak{q}}\) if \(\alpha - \beta \in \mathfrak{q}\). Via the identification described below, we use the same congruence notation for \(\alpha \in K^{\times}\) with \(v_{\mathfrak{q}}(\alpha) = 0\), interpreting \(\alpha\) through its residue class. Congruences between rational integers modulo a positive integer have the usual meaning. For a rational prime \(l\), ``modulo \(l\)'' means ``modulo \((l)\)''.

For an extension \(L/K\) of number fields and a prime ideal \(\mathfrak{q}_{L}\) of \(L\) lying above a prime ideal \(\mathfrak{q}_{K}\) of \(K\), we denote by \(f(\mathfrak{q}_{L}/\mathfrak{q}_{K})\) the residue degree.

For a prime ideal \(\mathfrak{q}\) of a number field \(K\), we denote its absolute norm by \(\Norm \mathfrak{q} = |\mathcal{O}_{K}/\mathfrak{q}|\). The subscripted symbol \(\Norm_{K/F}\) is the norm map mentioned above, and the unsubscripted symbol \(\Norm\) is the absolute norm.

For a number field \(K\), let \(d_{K}\) denote its discriminant.

For a number field \(K\), let \(\Cl(K)\) denote the ordinary ideal class group and \(h(K) = |\Cl(K)|\) the class number.
We also let \(\Cl^{+}(K)\) denote the narrow class group and \(h^{+}(K) = |\Cl^{+}(K)|\) the narrow class number.
Here the narrow class group is the quotient of the group of all nonzero fractional ideals by the subgroup of all principal ideals generated by totally positive nonzero elements.
For imaginary quadratic fields, the condition of total positivity is vacuous, and \(\Cl^{+}(K) = \Cl(K)\).
For a nonzero fractional ideal \(\mathfrak{a}\) of a number field \(K\), we denote by \([\mathfrak{a}]\) its class in the class group.

\subsection{Residue fields and \(p\)-th power residues}

For a number field \(K\) and a prime ideal \(\mathfrak{q}\) of \(K\), we denote by \(\bar{\alpha}\) the residue class of \(\alpha \in \mathcal{O}_{K}\) in the residue field \(\mathcal{O}_{K}/\mathfrak{q}\). When there is no danger of confusion, we denote \(\alpha\) and its residue class \(\bar{\alpha}\) by the same symbol. More generally, when \(\alpha \in K^{\times}\) satisfies \(v_{\mathfrak{q}}(\alpha) = 0\), the residue class of \(\alpha\) at \(\mathfrak{q}\) is well-defined as an element of \((\mathcal{O}_{K}/\mathfrak{q})^{\times}\), and we also denote it by the same symbol. In particular, when \(x \in \QQ^{\times}\) and a prime \(q\) divides neither the numerator nor the denominator of \(x\), we identify \(x\) with its residue class in \(\FF_{q}^{\times}\).

For a number field \(K\), a prime ideal \(\mathfrak{q}\) of \(K\), a prime \(p\), and \(a \in \mathcal{O}_{K}\) with \(\mathfrak{q} \nmid a\), we say that \(a\) is a \(p\)-th power residue modulo \(\mathfrak{q}\) if its residue class \(\bar{a}\) belongs to
\[
(\mathcal{O}_{K}/\mathfrak{q})^{\times p},
\]
and otherwise we say that \(a\) is a \(p\)-th power non-residue modulo \(\mathfrak{q}\). By the above identification, the same definition applies to \(a \in K^{\times}\) with \(v_{\mathfrak{q}}(a) = 0\). In particular, when \(x \in \QQ^{\times}\) and a prime \(q\) divides neither the numerator nor the denominator of \(x\), we write \(x \in \FF_{q}^{\times p}\) to indicate that \(x\) is a \(p\)-th power residue modulo \(q\), and \(x \notin \FF_{q}^{\times p}\) to indicate that \(x\) is a \(p\)-th power non-residue modulo \(q\).

\subsection{Square roots, \(p\)-th roots, and roots of unity}

For a positive real number \(x\), the symbol \(\sqrt{x}\) denotes the positive square root of \(x\). For a non-square integer \(D\), the field \(\QQ(\sqrt{D})\) does not depend on which root of \(X^{2} - D\) is taken as \(\sqrt{D}\).

For a prime \(p\), we denote by \(\mu_{p}\) the group of all \(p\)-th roots of unity in \(\overline{\QQ}\), and by \(\zeta_{p}\) a primitive \(p\)-th root of unity.

For \(a \in \overline{\QQ}^{\times}\) and a prime \(p\), we fix an element \(\alpha \in \overline{\QQ}\) with \(\alpha^{p} = a\) and denote it by \(\sqrt[p]{a}\). When \(a\) is a positive real number, we choose \(\sqrt[p]{a}\) to be the positive real root. Since the set of all roots of \(X^{p} - a\) is \(\{\zeta\sqrt[p]{a} : \zeta \in \mu_{p}\}\), for a field \(K \subseteq \overline{\QQ}\) with \(\mu_{p} \subseteq K\), the field \(K(\sqrt[p]{a})\) does not depend on the choice of \(\sqrt[p]{a}\).

\subsection{Fundamental units of real quadratic fields}

For a real quadratic field \(F\), Dirichlet's unit theorem implies that there exists a unique \(\varepsilon \in \mathcal{O}_{F}^{\times}\) such that
\[
\mathcal{O}_{F}^{\times} = \{\pm 1\} \times \langle \varepsilon \rangle, \qquad \varepsilon > 1.
\]
We call this \(\varepsilon\) the fundamental unit of \(F\). Throughout the paper, fundamental units are always taken with the normalization \(\varepsilon > 1\).

\section{Preliminaries}\label{sec:preliminaries}

In this section, we collect auxiliary results that hold in a general setting and are used in the proof of Theorem~\ref{thm:main}.
We first recall standard facts from Kummer theory that are used repeatedly in this paper (Subsection~\ref{subsec:kummer-theory}),
the definition and basic properties of Frobenius automorphisms together with Chebotarev's density theorem (Theorem~\ref{thm:chebotarev}),
and a criterion for \(p\)-th power residuosity in Kummer extensions (Proposition~\ref{prop:power-residue-criterion}).
We then give a consequence of genus theory controlling the \(2\)-part of the class number (Lemma~\ref{lem:genus-two-part}),
a lemma on the residuosity of the fundamental unit of a real quadratic field (Lemma~\ref{lem:unit-residuosity-transfer}),
and a criterion for the existence of an element of order \(n\) in the class group of a real quadratic field (Theorem~\ref{thm:class-group-criterion}).
Finally, we record a lemma based on Siegel's theorem that yields infinitely many distinct quadratic fields (Lemma~\ref{lem:siegel-quadratic-fields}).

The symbols \(F\), \(K\), \(L\), \(n\), \(p\), \(p_{i}\), \(l_{i}\), and so on used in this section are defined afresh in each statement,
and they are independent of the symbols fixed from Section~\ref{sec:construction} onward.

\subsection{Kummer theory}\label{subsec:kummer-theory}

We recall standard facts about Kummer extensions obtained by adjoining \(p\)-th roots to a field containing the \(p\)-th roots of unity, restricting ourselves to the case in which \(p\) is prime.
For statements and proofs in the case of adjoining general \(n\)-th roots, see \cite[Chapter~VI, \S\S6, 8, 9]{Lang2002} and \cite[Chapter~IV, \S3]{Neukirch1999}.

Throughout this subsection, let \(p\) be a prime and \(K\) a subfield of \(\overline{\QQ}\) with \(\mu_{p} \subseteq K\).
Since \(K\) has characteristic \(0\), no additional hypothesis on the characteristic is needed below.

For a field \(K\) and a prime \(p\), the quotient \(K^{\times}/K^{\times p}\) is an abelian group in which every element satisfies \(x^{p} = 1\), and hence it is naturally an \(\FF_{p}\)-vector space.
We say that the classes of \(a_{1}, \dots, a_{r} \in K^{\times}\) are linearly independent over \(\FF_{p}\) in \(K^{\times}/K^{\times p}\) if, whenever
\[
a_{1}^{u_{1}} \cdots a_{r}^{u_{r}} \in K^{\times p}
\]
for integers \(u_{1}, \dots, u_{r}\), one has
\[
u_{1} \equiv \dots \equiv u_{r} \equiv 0 \pmod{p}.
\]
In particular, if the classes are linearly independent, then no \(a_{k}\) belongs to \(K^{\times p}\).

\begin{proposition}\label{prop:kummer-basic}
Let \(a \in K^{\times}\), and put \(\alpha = \sqrt[p]{a}\) and \(L = K(\alpha)\). Then the following hold.
\begin{enumerate}[label=\textrm{(\roman*)}]
\item The field \(L\) is the splitting field of \(X^{p} - a\) over \(K\), and \(L/K\) is a Galois extension.
Moreover, \(L\) depends only on the class \(aK^{\times p} \in K^{\times}/K^{\times p}\) of \(a\).
In particular, \(K(\sqrt[p]{ab^{p}}) = K(\sqrt[p]{a})\) for any \(b \in K^{\times}\).
\item The polynomial \(X^{p} - a\) is irreducible over \(K\) if and only if \(a \notin K^{\times p}\).
Consequently, \([L:K] \in \{1, p\}\), and \([L:K] = p\) if and only if \(a \notin K^{\times p}\).
\item If \(a \notin K^{\times p}\), then \(L/K\) is a cyclic extension of degree \(p\), and the map
\[
\Gal(L/K) \longrightarrow \mu_{p}, \qquad \sigma \longmapsto \frac{\sigma(\alpha)}{\alpha}
\]
is an isomorphism of groups. In particular, there exists a unique \(\tau \in \Gal(L/K)\) with \(\tau(\alpha) = \zeta_{p}\alpha\), and \(\Gal(L/K) = \langle \tau \rangle\).
\end{enumerate}
\end{proposition}

\begin{proof}
(i) The set of all roots of \(X^{p} - a\) is \(\{\zeta\alpha : \zeta \in \mu_{p}\}\), and since \(\mu_{p} \subseteq K\), all of these roots belong to \(L = K(\alpha)\).
Hence \(L\) is the splitting field of \(X^{p} - a\) over \(K\), and since \(K\) has characteristic \(0\), \(L/K\) is a Galois extension.
For \(b \in K^{\times}\), \((b\alpha)^{p} = ab^{p}\), so that there exists \(\zeta \in \mu_{p}\) with \(\sqrt[p]{ab^{p}} = \zeta b\alpha\); consequently \(K(\sqrt[p]{ab^{p}}) = K(b\alpha) = K(\alpha)\).

(ii) If \(a \in K^{\times p}\), then \(X^{p} - a\) has a root in \(K\), and it is reducible since \(p \ge 2\).
Conversely, if \(a \notin K^{\times p}\), then \(X^{p} - a\) is irreducible over \(K\) by \cite[Chapter~VI, Theorem~9.1]{Lang2002}.
(Since \(p\) is prime, the additional condition in the cited theorem for exponents divisible by \(4\) does not arise.)
The assertion about \([L:K]\) follows because \([L:K]\) is the degree of the minimal polynomial of \(\alpha\) over \(K\).

(iii) For \(\sigma \in \Gal(L/K)\), the element \(\sigma(\alpha)\) is a root of \(X^{p} - a\), and thus \(\sigma(\alpha)/\alpha \in \mu_{p}\).
For \(\sigma, \sigma' \in \Gal(L/K)\), write \(\sigma'(\alpha) = \zeta'\alpha\) with \(\zeta' \in \mu_{p} \subseteq K\); then
\[
\frac{\sigma\sigma'(\alpha)}{\alpha} = \frac{\sigma(\zeta'\alpha)}{\alpha} = \zeta' \cdot \frac{\sigma(\alpha)}{\alpha},
\]
so that the map is a group homomorphism. Since \(L = K(\alpha)\), the equality \(\sigma(\alpha) = \alpha\) implies \(\sigma = \id_{L}\), which shows that the map is injective.
By (ii), \(|\Gal(L/K)| = p = |\mu_{p}|\), and thus the map is bijective; in particular, \(\Gal(L/K)\) is a cyclic group of order \(p\).
For the last assertion, it suffices to take \(\tau\) to be the preimage of \(\zeta_{p}\).
\end{proof}

\begin{lemma}\label{lem:kummer-degree}
Let \(\Gamma\) be a subgroup of \(K^{\times}\) with \(K^{\times p} \subseteq \Gamma \subseteq K^{\times}\) and \((\Gamma : K^{\times p}) < \infty\), and put
\[
L = K\bigl(\sqrt[p]{\Gamma}\bigr) = K\bigl(\sqrt[p]{a} : a \in \Gamma\bigr).
\]
Then \(L/K\) is a finite abelian extension, and \([L:K] = (\Gamma : K^{\times p})\).
\end{lemma}

\begin{proof}
See \cite[Chapter~VI, Theorem~8.1]{Lang2002}; see also \cite[Chapter~IV, \S3]{Neukirch1999}.
\end{proof}

We use the preceding lemma in the following form.

\begin{corollary}\label{cor:kummer-independent}
Suppose that the classes of \(a_{1}, \dots, a_{r} \in K^{\times}\) are linearly independent over \(\FF_{p}\) in \(K^{\times}/K^{\times p}\).
For each \(k \in \{1, \dots, r\}\), set \(\alpha_{k} = \sqrt[p]{a_{k}}\), and let \(L = K(\alpha_{1}, \dots, \alpha_{r})\). Then the following hold.
\begin{enumerate}[label=\textrm{(\roman*)}]
\item \([L:K] = p^{r}\). Moreover, for each \(u = (u_{1}, \dots, u_{r}) \in \FF_{p}^{\oplus r}\), there exists a unique \(\tau_{u} \in \Gal(L/K)\) such that
\[
\tau_{u}(\alpha_{k}) = \zeta_{p}^{u_{k}}\alpha_{k} \qquad (k = 1, \dots, r),
\]
and \(u \mapsto \tau_{u}\) gives a group isomorphism \(\FF_{p}^{\oplus r} \to \Gal(L/K)\).
\item For each \(k\), \(\tau_{u}|_{K(\alpha_{k})} = \id_{K(\alpha_{k})}\) if and only if \(u_{k} = 0\). In particular,
\[
\Gal\bigl(L/K(\alpha_{k})\bigr) = \{\tau_{u} : u \in \FF_{p}^{\oplus r},\ u_{k} = 0\}.
\]
\item Let \(F\) be a subfield of \(K\) such that \(K/F\) is a finite Galois extension and \(a_{1}, \dots, a_{r} \in F\).
Then \(L/F\) is a finite Galois extension. For \(\sigma \in \Gal(L/F)\), write \(\sigma(\zeta_{p}) = \zeta_{p}^{\chi(\sigma)}\) with \(\chi(\sigma) \in \FF_{p}^{\times}\); then, for any \(u \in \FF_{p}^{\oplus r}\),
\[
\sigma\tau_{u}\sigma^{-1} = \tau_{\chi(\sigma)u},
\]
where \(\chi(\sigma)u = (\chi(\sigma)u_{1}, \dots, \chi(\sigma)u_{r})\).
\end{enumerate}
\end{corollary}

\begin{proof}
(i) Let \(\Gamma\) be the subgroup of \(K^{\times}\) generated by \(a_{1}, \dots, a_{r}\) and \(K^{\times p}\). Then \(K^{\times p} \subseteq \Gamma \subseteq K^{\times}\), and \(L = K(\sqrt[p]{\Gamma})\).
By the linear-independence hypothesis, the classes of \(a_{1}, \dots, a_{r}\) form a basis of the \(\FF_{p}\)-vector space \(\Gamma/K^{\times p}\), and \((\Gamma : K^{\times p}) = p^{r}\).
Hence \([L:K] = p^{r}\) by Lemma~\ref{lem:kummer-degree}.
For \(\sigma \in \Gal(L/K)\), the element \(\sigma(\alpha_{k})\) is a root of \(X^{p} - a_{k}\), and therefore there exists a unique \(u_{k} \in \FF_{p}\) with \(\sigma(\alpha_{k}) = \zeta_{p}^{u_{k}}\alpha_{k}\).
As in the proof of Proposition~\ref{prop:kummer-basic}~(iii), the map
\[
\Gal(L/K) \longrightarrow \FF_{p}^{\oplus r}, \qquad
\sigma \longmapsto (u_{1}, \dots, u_{r})
\]
is a group homomorphism, and it is injective since \(L = K(\alpha_{1}, \dots, \alpha_{r})\).
Since \(|\Gal(L/K)| = p^{r} = |\FF_{p}^{\oplus r}|\), this map is an isomorphism. Writing its inverse as \(u \mapsto \tau_{u}\) gives (i).

(ii) Since \(K(\alpha_{k})\) is generated over \(K\) by \(\alpha_{k}\), the condition \(\tau_{u}|_{K(\alpha_{k})} = \id_{K(\alpha_{k})}\) is equivalent to \(\tau_{u}(\alpha_{k}) = \alpha_{k}\), which is equivalent to \(\zeta_{p}^{u_{k}} = 1\), hence to \(u_{k} = 0\).

(iii) Since \(K/F\) is a finite Galois extension, \(K\) is the splitting field over \(F\) of some \(g(X) \in F[X]\).
Since each \(a_{k}\) lies in \(F\) and all the roots of \(X^{p} - a_{k}\) belong to \(\mu_{p}\alpha_{k} \subseteq L\), the field \(L\) is the splitting field over \(F\) of \(g(X)\prod_{k=1}^{r}(X^{p} - a_{k})\), and \(L/F\) is a finite Galois extension.
Fix \(\sigma \in \Gal(L/F)\) and \(k\). Since \(\sigma^{-1}(\alpha_{k})\) is a root of \(X^{p} - a_{k}\), we can write \(\sigma^{-1}(\alpha_{k}) = \zeta\alpha_{k}\) for some \(\zeta \in \mu_{p}\). Hence
\[
\sigma\tau_{u}\sigma^{-1}(\alpha_{k})
= \sigma\bigl(\tau_{u}(\zeta\alpha_{k})\bigr)
= \sigma\bigl(\zeta\zeta_{p}^{u_{k}}\alpha_{k}\bigr)
= \zeta_{p}^{\chi(\sigma)u_{k}} \cdot \sigma(\zeta\alpha_{k})
= \zeta_{p}^{\chi(\sigma)u_{k}}\alpha_{k}.
\]
Here we used the fact that \(\tau_{u}\) fixes \(\zeta \in \mu_{p} \subseteq K\).
Since \(\sigma\tau_{u}\sigma^{-1} \in \Gal(L/K)\), the uniqueness in (i) gives \(\sigma\tau_{u}\sigma^{-1} = \tau_{\chi(\sigma)u}\).
\end{proof}

Finally, we record a basic fact on the non-ramification of prime ideals in Kummer extensions.

\begin{lemma}\label{lem:kummer-unramified}
Let \(K\) be a number field with \(\mu_{p} \subseteq K\), let \(a \in \mathcal{O}_{K} \setminus \{0\}\), let \(L = K(\sqrt[p]{a})\), and let \(\mathfrak{q}\) be a prime ideal of \(K\) with \(\mathfrak{q} \nmid pa\).
Then \(\mathfrak{q}\) is unramified in \(L/K\).
\end{lemma}

\begin{proof}
If \([L:K] = 1\), there is nothing to prove. Suppose that \([L:K] = p\). Then the minimal polynomial of \(\alpha = \sqrt[p]{a} \in \mathcal{O}_{L}\) over \(K\) is \(X^{p} - a\) (Proposition~\ref{prop:kummer-basic}~(ii)),
and \(1, \alpha, \dots, \alpha^{p - 1}\) form a \(K\)-basis of \(L\) contained in \(\mathcal{O}_{L}\).
The discriminant of this basis equals the discriminant of \(X^{p} - a\), which is \(p^{p}a^{p - 1}\) up to sign.
Let \(\mathfrak{d}_{L/K}\) be the relative discriminant ideal of \(L/K\), that is, the ideal of \(\mathcal{O}_{K}\) generated by the discriminants of all \(K\)-bases of \(L\) contained in \(\mathcal{O}_{L}\) (see \cite[Chapter~III, \S2]{Neukirch1999}).
By definition, the discriminant of the above basis belongs to \(\mathfrak{d}_{L/K}\), so that \(\mathfrak{q} \nmid pa\) implies \(\mathfrak{q} \nmid \mathfrak{d}_{L/K}\).
By Dedekind's discriminant theorem \cite[Chapter~III, \S2]{Neukirch1999}, only the prime ideals dividing \(\mathfrak{d}_{L/K}\) can ramify in \(L/K\), and hence \(\mathfrak{q}\) is unramified in \(L/K\).
\end{proof}

\subsection{Criterion for \(p\)-th power residuosity}\label{subsec:power-residue-criterion}

We begin with a basic lemma on \(n\)-th powers when the extension degree is coprime to \(n\).

\begin{lemma}\label{lem:coprime-degree-nth-powers}
Let \(K/F\) be a finite extension of fields and let \(n\) be a positive integer. Set \(f = [K:F]\) and assume that \(\gcd(f, n) = 1\).
Then we have
\[
K^{\times n} \cap F^{\times} = F^{\times n}.
\]
\end{lemma}

\begin{proof}
The inclusion \(F^{\times n} \subseteq K^{\times n} \cap F^{\times}\) is clear.
To show the reverse inclusion, let \(a \in K^{\times n} \cap F^{\times}\).
There exists \(\alpha \in K^{\times}\) with \(\alpha^{n} = a\).
Taking norms from \(K\) to \(F\), we obtain
\[
\Norm_{K/F}(\alpha)^{n} = a^{f}.
\]
Since \(\gcd(f, n) = 1\), there exist integers \(x, y\) with \(fx + ny = 1\). Therefore,
\[
\begin{aligned}
a
&= a^{fx + ny}
= (a^{f})^{x}(a^{y})^{n} \\
&= \bigl(\Norm_{K/F}(\alpha)^{x} a^{y}\bigr)^{n}.
\end{aligned}
\]
Since \(\Norm_{K/F}(\alpha)^{x} a^{y} \in F^{\times}\), it follows that \(a \in F^{\times n}\).
\end{proof}

Let \(L/K\) be a finite Galois extension of number fields, and let \(\Gal(L/K)\) be its Galois group.
Let \(\mathfrak{q}_{K}\) be a prime ideal of \(K\) unramified in \(L\),
and let \(\mathfrak{q}_{L}\) be a prime ideal of \(L\) lying above \(\mathfrak{q}_{K}\).
Then there exists a unique element \(\Frob{L/K}{\mathfrak{q}_{L}}\) of \(\Gal(L/K)\) such that
\[
\Frob{L/K}{\mathfrak{q}_{L}}(x) \equiv x^{\Norm \mathfrak{q}_{K}} \pmod{\mathfrak{q}_{L}}\quad\text{for all \(x \in \mathcal{O}_{L}\)}.
\]
This element \(\Frob{L/K}{\mathfrak{q}_{L}}\) is called the Frobenius automorphism of \(\mathfrak{q}_{L}\).

For \(\sigma \in \Gal(L/K)\), we have
\[
\Frob{L/K}{\mathfrak{q}_{L}^{\sigma}} = \sigma \cdot \Frob{L/K}{\mathfrak{q}_{L}} \cdot \sigma^{-1},
\]
and thus the conjugacy class
\[
\left\{\sigma \cdot \Frob{L/K}{\mathfrak{q}_{L}} \cdot \sigma^{-1} : \sigma \in \Gal(L/K)\right\}
\]
of \(\Frob{L/K}{\mathfrak{q}_{L}}\) in \(\Gal(L/K)\) is determined by \(\mathfrak{q}_{K}\) alone, independently of the choice of the prime ideal of \(L\) lying above \(\mathfrak{q}_{K}\).
This conjugacy class is called the Frobenius class of \(\mathfrak{q}_{K}\).

When \(L/K\) is an abelian extension, that is, when \(\Gal(L/K)\) is an abelian group,
the Frobenius automorphisms of all the prime ideals of \(L\) lying above \(\mathfrak{q}_{K}\) coincide, and the Frobenius class is a singleton.
In this case, we write \(\Frob{L/K}{\mathfrak{q}_{K}}\) for \(\Frob{L/K}{\mathfrak{q}_{L}}\).

The distribution of Frobenius classes is governed by Chebotarev's density theorem.
A set \(\mathcal{P}\) of prime ideals of \(K\) is said to have Dirichlet density \(\delta\) if
\[
\lim_{s \to 1^{+}} \frac{\sum_{\mathfrak{q} \in \mathcal{P}} (\Norm\mathfrak{q})^{-s}}{\log\dfrac{1}{s - 1}} = \delta
\]
(see \cite[Chapter~VII, \S13]{Neukirch1999}).
A set with positive Dirichlet density is infinite.

\begin{theorem}[Chebotarev's density theorem]\label{thm:chebotarev}
Let \(L/K\) be a finite Galois extension of number fields, let \(G = \Gal(L/K)\), and let \(\mathcal{C} \subseteq G\) be a conjugacy class of \(G\).
Let \(\mathcal{P}_{L/K}(\mathcal{C})\) be the set of all prime ideals \(\mathfrak{q}_{K}\) of \(K\) unramified in \(L/K\) whose Frobenius class is \(\mathcal{C}\).
Then \(\mathcal{P}_{L/K}(\mathcal{C})\) has Dirichlet density
\[
\frac{|\mathcal{C}|}{|G|}.
\]
In particular, \(\mathcal{P}_{L/K}(\mathcal{C})\) is infinite.
\end{theorem}

\begin{proof}
See \cite[Chapter~VII, Theorem~(13.4)]{Neukirch1999}.
\end{proof}

We use the following specialization to \(K = \QQ\).

\begin{corollary}\label{cor:chebotarev-rational}
Let \(L/\QQ\) be a finite Galois extension, let \(\mathcal{C}\) be a conjugacy class of \(\Gal(L/\QQ)\), and let \(S\) be a finite set of rational primes.
Then there exist infinitely many rational primes \(l \notin S\) that are unramified in \(L/\QQ\) and whose Frobenius class is \(\mathcal{C}\).
Moreover, if \(K\) is an intermediate field of \(L/\QQ\) such that \(K/\QQ\) is a Galois extension and every element of \(\mathcal{C}\) acts trivially on \(K\), then every such \(l\) splits completely in \(K/\QQ\).
\end{corollary}

\begin{proof}
The first assertion follows from Theorem~\ref{thm:chebotarev} and the finiteness of \(S\).
We prove the second assertion. Let \(\mathfrak{L}\) be a prime ideal of \(L\) lying above \(l\), and put \(\mathfrak{l} = \mathfrak{L} \cap \mathcal{O}_{K}\).
Since \(K/\QQ\) is a Galois extension, the restriction of \(\Frob{L/\QQ}{\mathfrak{L}}\) to \(K\) equals \(\Frob{K/\QQ}{\mathfrak{l}}\).
By assumption, \(\Frob{L/\QQ}{\mathfrak{L}} \in \mathcal{C}\) restricts to the identity on \(K\), so that \(\Frob{K/\QQ}{\mathfrak{l}} = \id_{K}\).
Since \(l\) is unramified in \(K/\QQ\) and the decomposition group of \(\mathfrak{l}\) is generated by \(\Frob{K/\QQ}{\mathfrak{l}}\), this decomposition group is trivial, and \(l\) splits completely in \(K/\QQ\).
\end{proof}

\begin{proposition}\label{prop:power-residue-criterion}
Let \(F\) be a number field, \(\gamma \in \mathcal{O}_{F} \setminus \{0\}\), \(p\) a prime, and \(\zeta_{p}\) a primitive \(p\)-th root of unity.
Put \(K = F(\zeta_{p})\) and \(L = K(\sqrt[p]{\gamma})\), and assume that \(\gamma \notin K^{\times p}\).

Let \(\mathfrak{q}\) be a prime ideal of \(F\) dividing neither \(p\) nor \(\gamma\).
Let \(\mathfrak{q}_{K}\) be a prime ideal of \(K\) lying above \(\mathfrak{q}\).
Then \(\mathfrak{q}_{K}\) is unramified in \(L/K\). Furthermore, the following are equivalent.
\begin{enumerate}[label=\textrm{(\roman*)}]
\item The element \(\gamma\) is a \(p\)-th power residue modulo \(\mathfrak{q}\).
\item \(\Frob{L/K}{\mathfrak{q}_{K}} = \id_{L}\).
\item The prime ideal \(\mathfrak{q}_{K}\) splits completely in \(L/K\).
\end{enumerate}
\end{proposition}

\begin{proof}
Put \(\alpha = \sqrt[p]{\gamma}\).
Since \(\mu_{p} \subseteq K\) and \(\gamma \notin K^{\times p}\),
the extension \(L/K\) is cyclic of degree \(p\) by Proposition~\ref{prop:kummer-basic}~(iii).
In addition, since \(\mathfrak{q}_{K} \nmid p\gamma\),
the prime ideal \(\mathfrak{q}_{K}\) is unramified in \(L/K\) by Lemma~\ref{lem:kummer-unramified}.

First, put
\[
\kappa = \mathcal{O}_{F}/\mathfrak{q}, \qquad
\kappa' = \mathcal{O}_{K}/\mathfrak{q}_{K}, \qquad f = [\kappa':\kappa].
\]
Since \(K/F\) is a Galois extension and
\(f \mid [K:F] \mid p - 1\), we see that \(\gcd(f, p) = 1\).
Let \(\bar{\gamma} \in \kappa^{\times}\) be the residue class of \(\gamma\) modulo \(\mathfrak{q}\).
By applying Lemma~\ref{lem:coprime-degree-nth-powers} to
the extension \(\kappa'/\kappa\) and the exponent \(p\), we obtain
\[
(\kappa')^{\times p} \cap \kappa^{\times} = \kappa^{\times p}.
\]
Therefore, condition (i) is equivalent to
\(\bar{\gamma} \in (\kappa')^{\times p}\).

Let \(\mathfrak{q}_{L}\) be a prime ideal of \(L\) lying above \(\mathfrak{q}_{K}\),
and put \(\sigma = \Frob{L/K}{\mathfrak{q}_{K}}\).
In what follows, \(\bar{\alpha}\) denotes the residue class of \(\alpha\) modulo \(\mathfrak{q}_{L}\).
Since \(\mathfrak{q}_{K} \nmid p\gamma\),
the roots of \(X^{p} - \gamma\) have mutually distinct residue classes modulo \(\mathfrak{q}_{L}\).
Thus, by the definition of the Frobenius automorphism,
\(\sigma(\alpha) = \alpha\) if and only if \(\bar{\alpha}^{|\kappa'|} = \bar{\alpha}\),
which is equivalent to \(\bar{\alpha} \in \kappa'\).
Furthermore, since \(\mu_{p} \subseteq K\),
if \(X^{p} - \bar{\gamma}\) has a root in \(\kappa'\),
then all of its roots belong to \(\kappa'\).
Hence \(\bar{\alpha} \in \kappa'\) is equivalent to
\(\bar{\gamma} \in (\kappa')^{\times p}\).
Since \(L = K(\alpha)\), the equivalence of (i) and (ii) follows from the above.

Finally, since the decomposition group of an unramified prime ideal is generated by the Frobenius automorphism,
(ii) and (iii) are equivalent.
\end{proof}

\subsection{The \(2\)-part of the class number via genus theory}\label{subsec:genus-theory}

The \(2\)-part of the class number can be controlled using Gauss's genus theory, independently of Kummer theory.

\begin{lemma}[The \(2\)-part via genus theory]\label{lem:genus-two-part}
Let \(K\) be a quadratic field, and put \(s = \omega(d_{K})\). Then the \(2\)-rank of \(\Cl^{+}(K)\) is \(s - 1\); that is,
\[
\dim_{\FF_{2}}\bigl(\Cl^{+}(K)/\Cl^{+}(K)^{2}\bigr) = s - 1,
\]
and therefore
\[
2^{s - 1} \mid h^{+}(K).
\]
In particular:
\begin{enumerate}[label=\textrm{(\roman*)}]
\item If \(K\) is an imaginary quadratic field, then \(2^{s - 1} \mid h(K)\).
\item If \(K\) is a real quadratic field and \(s \ge 2\), then \(2^{s - 2} \mid h(K)\).
\end{enumerate}
\end{lemma}

\begin{proof}
The first equality is a well-known fact in Gauss's genus theory (see, for example, \cite[\S2.2]{Lemmermeyer2000}).
Therefore \(2^{s - 1}\) divides \(h^{+}(K)\).
For imaginary quadratic fields, \(\Cl^{+}(K) = \Cl(K)\), and (i) follows.
For real quadratic fields, the kernel of the natural surjection
\[
\Cl^{+}(K) \longrightarrow \Cl(K)
\]
onto the ordinary class group has order \(1\) or \(2\) (see, for example, \cite[Corollary~2.7]{Lemmermeyer2000}), so that
\[
h^{+}(K) \in \{h(K), 2h(K)\}.
\]
Thus \(v_{2}(h(K)) \ge s - 2\), and (ii) follows.
\end{proof}

\subsection{Transfer of residuosity to the fundamental unit}\label{subsec:unit-residuosity}

The fundamental unit of a real quadratic field cannot be determined explicitly in general;
however, if an explicit unit of infinite order is available, the residuosity of the fundamental unit can be deduced from that of the explicit unit.

\begin{lemma}[Residuosity of the fundamental unit via an explicit unit]\label{lem:unit-residuosity-transfer}
Let \(F\) be a real quadratic field, \(\varepsilon\) its fundamental unit, and \(p\) an odd prime. Let \(\eta \in \mathcal{O}_{F}^{\times}\) be a unit of infinite order. Suppose that there exist two prime ideals \(\mathfrak{l}\), \(\mathfrak{q}\) of \(F\) such that
\[
\eta \in (\mathcal{O}_{F}/\mathfrak{l})^{\times p}, \qquad
\eta \notin (\mathcal{O}_{F}/\mathfrak{q})^{\times p}.
\]
Then we have
\[
\varepsilon \in (\mathcal{O}_{F}/\mathfrak{l})^{\times p}.
\]
\end{lemma}

\begin{proof}
Since the unit group of a real quadratic field is \(\{\pm 1\} \times \langle \varepsilon \rangle\), we can write
\[
\eta = \pm\varepsilon^{\nu}
\]
for some \(\nu \in \ZZ \setminus \{0\}\). Since \(p\) is odd, \(-1 = (-1)^{p}\) is a \(p\)-th power. If \(p \mid \nu\), then \(\eta\) is a \(p\)-th power in \(F\), and thus it is also a \(p\)-th power modulo \(\mathfrak{q}\). This contradicts the assumption. Hence \(p \nmid \nu\).

Let
\[
G = (\mathcal{O}_{F}/\mathfrak{l})^{\times}
\]
be the multiplicative group of the finite field \(\mathcal{O}_{F}/\mathfrak{l}\), and consider the quotient \(G/G^{p}\) of \(G\) by its subgroup \(G^{p}\) of \(p\)-th powers. This quotient is cyclic of order \(1\) or \(p\).
Let \(\pi\colon G \to G/G^{p}\) be the natural quotient map. If we write \(G/G^{p}\) additively, then
\[
\pi(\eta) = \nu\pi(\varepsilon).
\]
By assumption, \(\pi(\eta) = 0\), and since \(\nu\) is invertible modulo \(p\), we obtain \(\pi(\varepsilon) = 0\).
\end{proof}

\subsection{Criterion for the existence of an element of order \(n\) in the class group}\label{subsec:class-group-criterion}

The following theorem is the main tool for handling the odd part \(n\).

\begin{theorem}[Class group criterion for real quadratic fields]\label{thm:class-group-criterion}
Let \(D\) be a positive non-square integer, and put \(F = \QQ(\sqrt{D})\).
Let \(n \ge 3\) be an odd integer, and let \(p_{1}, p_{2}, \dots, p_{\omega(n)}\) be all the distinct prime factors of \(n\).
Suppose that integers \(X, Z\) satisfy
\begin{equation}\label{eq:XZ-equation}
X^{2} - D = 4XZ^{n},
\qquad \gcd(X, Z) = 1.
\end{equation}
Suppose moreover that, for each \(i \in \{1, 2, \dots, \omega(n)\}\), there exists a prime \(l_{i}\) satisfying the following conditions.
\begin{enumerate}
\item[\textrm{(C1)}] \(l_{i} \equiv 1 \pmod{p_{i}}\).
\item[\textrm{(C2)}] The integer \(X\) is a \(p_{i}\)-th power non-residue modulo \(l_{i}\).
\item[\textrm{(C3)}] The prime \(l_{i}\) divides \(Z\).
\item[\textrm{(C4)}] The fundamental unit \(\varepsilon\) of \(F\) is a \(p_{i}\)-th power residue modulo every prime ideal of \(F\) lying above \(l_{i}\).
\end{enumerate}
Then the class group of \(F\) contains an element of order \(n\).
\end{theorem}

\begin{proof}
Let \(\sigma\) be the nontrivial automorphism of \(F/\QQ\), and put
\[
\alpha = \frac{X - 2Z^{n} + \sqrt{D}}{2}.
\]
By~\eqref{eq:XZ-equation},
\[
\alpha + \alpha^{\sigma} = X - 2Z^{n},
\qquad
\alpha\alpha^{\sigma} = Z^{2n}.
\]
Therefore,
\[
\alpha^{2} - (X - 2Z^{n})\alpha + Z^{2n} = 0,
\]
and since \(\alpha\) is a root of a monic quadratic polynomial with integer coefficients, we conclude that \(\alpha \in \mathcal{O}_{F}\).

Suppose that there exists a prime ideal \(\mathfrak{P}\) dividing both \((\alpha)\) and \((\alpha^{\sigma})\), and put \(\mathfrak{P} \cap \ZZ = (r)\).
Since \(\mathfrak{P}\) divides \(\alpha\alpha^{\sigma} = Z^{2n}\) and \(\alpha + \alpha^{\sigma} = X - 2Z^{n}\), we see that \(r \mid Z\) and \(r \mid X\), which contradicts \(\gcd(X, Z) = 1\).
Hence \((\alpha)\) and \((\alpha^{\sigma})\) are coprime. If \(Z = 0\), then \(D = X^{2}\) is a square by~\eqref{eq:XZ-equation}, contrary to the assumption; hence \(Z \ne 0\), so that \(\alpha\alpha^{\sigma} = Z^{2n} \ne 0\). On the other hand,
\[
(\alpha)(\alpha^{\sigma}) = (Z)^{2n}.
\]
Consequently, by the unique factorization of ideals and the coprimality of \((\alpha)\) and \((\alpha^{\sigma})\), the exponent of each prime ideal appearing in \((\alpha)\) is a multiple of \(2n\).
Hence there exists an integral ideal \(\mathfrak{a}\) such that
\[
(\alpha) = \mathfrak{a}^{2n}.
\]

Fix \(i \in \{1, 2, \dots, \omega(n)\}\). By condition (C2), \(l_{i} \nmid X\), and by condition (C3) and~\eqref{eq:XZ-equation},
\[
D \equiv X^{2} \not\equiv 0 \pmod{l_{i}}.
\]
Note also that \(l_{i}\) is odd by condition (C1) and the fact that \(p_{i}\) is an odd prime.
Since \(l_{i} \nmid D\) and \(D\) is a square modulo \(l_{i}\), the prime \(l_{i}\) splits completely in \(F/\QQ\).
Choose a prime ideal \(\mathfrak{l}_{i}\) lying above \(l_{i}\) such that
\[
\sqrt{D} \equiv X \pmod{\mathfrak{l}_{i}}.
\]
Since \(l_{i}\) is odd, \(2\) is invertible modulo \(\mathfrak{l}_{i}\), and since \(Z^{n} \equiv 0 \pmod{\mathfrak{l}_{i}}\) by condition (C3), we find
\[
\alpha = \frac{X - 2Z^{n} + \sqrt{D}}{2} \equiv \frac{X + X}{2} = X \pmod{\mathfrak{l}_{i}}.
\]
Hence, by condition (C2), \(\alpha\) is a \(p_{i}\)-th power non-residue modulo \(\mathfrak{l}_{i}\). In particular, \(\alpha \not\equiv 0 \pmod{\mathfrak{l}_{i}}\).

If \(\mathfrak{a}^{2n/p_{i}}\) is principal, then there exist \(\beta \in \mathcal{O}_{F}\) and \(s \in \ZZ\) such that
\[
\mathfrak{a}^{2n/p_{i}} = (\beta),
\qquad
\alpha = \pm\varepsilon^{s}\beta^{p_{i}}.
\]
Since \(\alpha \not\equiv 0 \pmod{\mathfrak{l}_{i}}\), we get \(\beta \not\equiv 0 \pmod{\mathfrak{l}_{i}}\), and \(\beta^{p_{i}}\) is a \(p_{i}\)-th power residue modulo \(\mathfrak{l}_{i}\).
By condition (C4), \(\varepsilon\) is a \(p_{i}\)-th power residue modulo \(\mathfrak{l}_{i}\), and since \(p_{i}\) is odd, \(-1 = (-1)^{p_{i}}\) is also a \(p_{i}\)-th power.
Therefore the right-hand side is a \(p_{i}\)-th power residue modulo \(\mathfrak{l}_{i}\), contradicting the fact that \(\alpha\) is a \(p_{i}\)-th power non-residue modulo \(\mathfrak{l}_{i}\).
Hence
\[
[\mathfrak{a}]^{2n/p_{i}} \ne 1.
\]

Let \(d\) be the order of \([\mathfrak{a}]\). Then \(d \mid 2n\).
If \(n \nmid d\), then \(d \mid 2n/p_{i}\) for some \(i\), so that \([\mathfrak{a}]^{2n/p_{i}} = 1\), a contradiction.
Thus \(n \mid d\), and \(d = n\) or \(d = 2n\).
In the former case \([\mathfrak{a}]\) has order \(n\), and in the latter case \([\mathfrak{a}]^{2}\) has order \(n\).
\end{proof}

\begin{remark}\label{rem:class-group-criterion-xie-chao}
Theorem~\ref{thm:class-group-criterion} corresponds to the special case of
\cite[Theorem~2.3]{XieChao2020}
in which \(m = k = 1\) and \(Y = 1\) in the notation of that paper.
To make explicit where the condition on the fundamental unit enters the argument, we have given a direct proof of the special case needed here.
\end{remark}

\subsection{A consequence of Siegel's theorem}\label{subsec:siegel}

The following lemma is a consequence of Siegel's theorem on the finiteness of integral points on smooth algebraic curves of genus at least \(1\).

\begin{lemma}[{Xie--Chao \cite[Lemma~4.1]{XieChao2022}}]\label{lem:siegel-quadratic-fields}
Suppose that \(f(X) \in \QQ[X]\) is a polynomial of degree \(d \ge 3\) with distinct roots in \(\CC\),
and let \(\mathcal{U} \subseteq \ZZ\) be an infinite set of integers. Set
\[
\mathcal{E} = \{\QQ(\sqrt{f(t)}) : t \in \mathcal{U}\}.
\]
Then \(\mathcal{E}\) contains infinitely many distinct quadratic fields.
\end{lemma}

\section{Construction of pairs of real quadratic fields}\label{sec:construction}

Let \(e \ge 0\) be an integer and \(n \ge 3\) an odd integer, and let
\(p_{1}, p_{2}, \dots, p_{\omega(n)}\) be all the distinct prime factors of \(n\).
Since \(n\) is odd, each \(p_{i}\) is an odd prime.
Let \(m \ge 1\) be an integer.
All the statements in this section and in Sections~\ref{sec:divisibility}--\ref{sec:infinitude} are made under these assumptions and this notation,
together with the quantities \(H\), \(\lambda_{j, \pm}\), \(\xi_{j, k}\), \(l_{i, j}\), and \(q_{i, j}\), which will be fixed successively from this section on.
The general auxiliary results needed here are collected in Section~\ref{sec:preliminaries}.
From this section onward, we construct elements of order \(n\) in the class groups to handle the odd part, while simultaneously ensuring divisibility of the class numbers by \(2^{e}\) via genus theory.
First, put
\begin{equation}\label{eq:def-c}
c_{1} = 4m + 1, \qquad c_{2} = 4m - 1.
\end{equation}
This gives
\begin{equation}\label{eq:c-square-difference}
c_{1}^{2} - c_{2}^{2} = 16m.
\end{equation}

\begin{lemma}\label{lem:choice-of-H}
There exist an odd integer \(H\) and odd primes
\[
\lambda_{j, -},\ \lambda_{j, +},\ \xi_{j, 1}, \dots, \xi_{j, e}
\qquad (j \in \{1, 2\})
\]
satisfying the following conditions, where all the primes displayed are required to be mutually distinct
(when \(e = 0\), there are no primes \(\xi_{j, k}\)).
\begin{enumerate}[label=\textrm{(\roman*)}]
\item \(H > c_{1}\), \(H \equiv c_{1} \pmod{8}\), and \(\gcd(H, c_{1}c_{2}) = 1\).
\item \(v_{\lambda_{j, -}}(H^{n} - c_{j}) = 1\) and \(v_{\lambda_{j, +}}(H^{n} + c_{j}) = 1\) (\(j = 1, 2\)).
\item \(v_{\xi_{j, k}}(H^{n} - c_{j}) = 1\) (\(j = 1, 2\), \(1 \le k \le e\)).
\item None of these primes divides \(2nmc_{1}c_{2}\).
\end{enumerate}
\end{lemma}

\begin{proof}
For each \(j \in \{1, 2\}\), put
\[
g_{j, -}(X) = X^{n} - c_{j}, \qquad g_{j, +}(X) = X^{n} + c_{j}.
\]
By Schur's theorem (see \cite[pp.~40--41]{Schur1912}),
for each polynomial \(g_{j, \pm}(X)\), there exist infinitely many primes that divide \(g_{j, \pm}(x)\) for some integer \(x\).
Consequently, by excluding finitely many primes at each stage, for each \(j\) we can choose
\(e + 1\) primes \(\lambda_{j, -}, \xi_{j, 1}, \dots, \xi_{j, e}\) among the prime divisors of integer values of \(g_{j, -}\)
and one prime \(\lambda_{j, +}\) among the prime divisors of integer values of \(g_{j, +}\)
so that all the displayed primes are mutually distinct and none of them divides \(2nmc_{1}c_{2}\).
For each selected prime \(r\), if we write \(g\) for the corresponding polynomial,
then there exists an integer \(x_{r}\) such that
\[
g(x_{r}) \equiv 0 \pmod{r}.
\]
Here \(r \nmid c_{j}\), so that \(r \nmid x_{r}\); moreover, \(r \nmid n\), and hence
\[
g'(x_{r}) = n x_{r}^{n - 1} \not\equiv 0 \pmod{r}.
\]
By Hensel's lemma, among the residue classes modulo \(r^{2}\) lifting the residue class of \(x_{r}\) modulo \(r\), exactly one is a root of \(g\) modulo \(r^{2}\).
Choosing a lift \(y_{r}\) different from this one, we obtain
\[
v_{r}\bigl(g(y_{r})\bigr) = 1.
\]

The moduli \(8\), \(\rad(c_{1}c_{2})\), and \(r^{2}\) for all the selected primes \(r\) are pairwise coprime.
Hence, by the Chinese remainder theorem, there exists an integer \(H\) satisfying simultaneously
\[
\begin{aligned}
H &\equiv c_{1} \pmod{8}, \\
H &\equiv 1 \pmod{\rad(c_{1}c_{2})}, \\
H &\equiv y_{r} \pmod{r^{2}}
\quad\text{for every selected prime \(r\)}.
\end{aligned}
\]
Replacing \(H\) by a sufficiently large positive integer in the same residue class, we may also assume that \(H > c_{1}\); then all the conditions hold.
\end{proof}

In what follows, we fix \(H\), \(\lambda_{j, \pm}\), and \(\xi_{j, k}\) satisfying the conditions of Lemma~\ref{lem:choice-of-H}. For each \(j \in \{1, 2\}\), put
\begin{equation}\label{eq:def-abC}
a_{j} = H^{n} - c_{j}, \qquad b_{j} = H^{n} + c_{j}, \qquad C_{j} = a_{j}b_{j} = H^{2n} - c_{j}^{2}.
\end{equation}

\begin{lemma}\label{lem:properties-abC}
For each \(j \in \{1, 2\}\), the following hold.
\begin{enumerate}[label=\textrm{(\roman*)}]
\item \(C_{j} > 0\) and \(16 \mid C_{j}\).
\item \(\gcd(H, C_{j}c_{j}) = 1\).
\item \(v_{\lambda_{j, -}}(a_{j}) = v_{\lambda_{j, +}}(b_{j}) = 1\) and \(v_{\lambda_{j, -}}(b_{j}) = v_{\lambda_{j, +}}(a_{j}) = 0\).
\item \(v_{\xi_{j, k}}(a_{j}) = 1\) and \(v_{\xi_{j, k}}(b_{j}) = 0\) (\(1 \le k \le e\)).
\end{enumerate}
\end{lemma}

\begin{proof}
Since \(H > c_{1} \ge c_{j} > 0\), it follows that \(C_{j} = (H^{n} - c_{j})(H^{n} + c_{j}) > 0\). Next, since \(H \equiv c_{1} \pmod{8}\), we can write \(H = c_{1} + 8k_{0}\) with \(k_{0} \in \ZZ\), so that
\[
H^{2} = c_{1}^{2} + 16c_{1}k_{0} + 64k_{0}^{2} \equiv c_{1}^{2} \pmod{16}.
\]
Furthermore, since \(c_{1}\) is odd, \(c_{1}^{2} \equiv 1 \pmod{8}\), and therefore \(c_{1}^{4} \equiv 1 \pmod{16}\).
Since \(n\) is odd, we have
\[
H^{2n} \equiv (c_{1}^{2})^{n} = c_{1}^{2}\,(c_{1}^{4})^{(n - 1)/2} \equiv c_{1}^{2} \pmod{16}.
\]
By~\eqref{eq:c-square-difference}, \(c_{1}^{2} \equiv c_{2}^{2} \pmod{16}\), and hence \(16 \mid C_{j}\).

By Lemma~\ref{lem:choice-of-H}~(i), \(\gcd(H, c_{j}) = 1\), and
\[
\gcd(H, C_{j}) = \gcd(H, H^{2n} - c_{j}^{2}) = \gcd(H, c_{j}^{2}) = 1;
\]
hence (ii) holds.

From Lemma~\ref{lem:choice-of-H}~(ii), we obtain
\[
v_{\lambda_{j, -}}(a_{j}) = v_{\lambda_{j, +}}(b_{j}) = 1.
\]
Furthermore, \(b_{j} - a_{j} = 2c_{j}\), and \(\lambda_{j, \pm} \nmid 2c_{j}\) by Lemma~\ref{lem:choice-of-H}~(iv); hence neither \(\lambda_{j, -}\) nor \(\lambda_{j, +}\) divides both \(a_{j}\) and \(b_{j}\). It follows that
\[
v_{\lambda_{j, -}}(b_{j}) = v_{\lambda_{j, +}}(a_{j}) = 0,
\]
and (iii) holds.

Similarly, \(v_{\xi_{j, k}}(a_{j}) = 1\) by Lemma~\ref{lem:choice-of-H}~(iii).
Since \(\xi_{j, k} \nmid 2c_{j}\) by Lemma~\ref{lem:choice-of-H}~(iv),
the prime \(\xi_{j, k}\) cannot divide both \(a_{j}\) and \(b_{j}\), and
\[
v_{\xi_{j, k}}(b_{j}) = 0.
\]
Hence (iv) holds.
\end{proof}

For each \(j \in \{1, 2\}\), put
\begin{equation}\label{eq:def-genus-primes}
\mathcal{R}_{j}
= \{\lambda_{j, -}, \lambda_{j, +}\}
\cup \{\xi_{j, 1}, \dots, \xi_{j, e}\}.
\end{equation}
By Lemmas~\ref{lem:choice-of-H} and~\ref{lem:properties-abC},
the set \(\mathcal{R}_{j}\) consists of \(e + 2\) distinct odd primes,
and \(v_{r}(C_{j}) = 1\) for all \(r \in \mathcal{R}_{j}\).

For each \(i \in \{1, 2, \dots, \omega(n)\}\), let \(\zeta_{p_{i}}\) be a primitive \(p_{i}\)-th root of unity, and put \(K_{i} = \QQ(\zeta_{p_{i}})\). Furthermore, for each \(j \in \{1, 2\}\), put
\begin{equation}\label{eq:def-AE}
E_{j} = \frac{a_{j}}{b_{j}} = \frac{H^{n} - c_{j}}{H^{n} + c_{j}}, \qquad
A_{j} = \frac{4c_{j}^{4}}{C_{j}}.
\end{equation}

We use the notion of linear independence over \(\FF_{p}\) in \(K^{\times}/K^{\times p}\) defined in Subsection~\ref{subsec:kummer-theory}.

\begin{lemma}\label{lem:AE-independent}
For any \(i \in \{1, 2, \dots, \omega(n)\}\) and \(j \in \{1, 2\}\), the classes of \(A_{j}\) and \(E_{j}\) are linearly independent over \(\FF_{p_{i}}\) in \(K_{i}^{\times}/K_{i}^{\times p_{i}}\).
\end{lemma}

\begin{proof}
First, let \(u, v \in \ZZ\), and suppose that
\[
A_{j}^{u}E_{j}^{v} \in \QQ^{\times p_{i}}.
\]
By Lemma~\ref{lem:choice-of-H}~(iv), \(v_{\lambda_{j, \pm}}(4c_{j}^{4}) = 0\), and hence, together with Lemma~\ref{lem:properties-abC}~(iii),
\[
\begin{aligned}
v_{\lambda_{j, -}}(A_{j}) &= -1, &
v_{\lambda_{j, -}}(E_{j}) &= 1, \\
v_{\lambda_{j, +}}(A_{j}) &= -1, &
v_{\lambda_{j, +}}(E_{j}) &= -1.
\end{aligned}
\]
Consequently,
\[
-u + v \equiv 0 \pmod{p_{i}}, \qquad -u - v \equiv 0 \pmod{p_{i}}.
\]
Since \(n\) is odd, \(p_{i}\) is an odd prime, so we obtain \(u \equiv v \equiv 0 \pmod{p_{i}}\). Hence the classes of \(A_{j}\) and \(E_{j}\) are linearly independent in \(\QQ^{\times}/\QQ^{\times p_{i}}\).

In addition, since \([K_{i}:\QQ] = p_{i} - 1\), Lemma~\ref{lem:coprime-degree-nth-powers} gives
\[
K_{i}^{\times p_{i}} \cap \QQ^{\times} = \QQ^{\times p_{i}}.
\]
Thus the classes are linearly independent in \(K_{i}^{\times}/K_{i}^{\times p_{i}}\) as well.
\end{proof}

\begin{lemma}\label{lem:choice-of-auxiliary-primes}
For each \(i \in \{1, 2, \dots, \omega(n)\}\) and \(j \in \{1, 2\}\), we can choose rational primes \(l_{i, j}\) and \(q_{i, j}\) satisfying the following conditions.
\begin{enumerate}[label=\textrm{(\roman*)}]
\item \(l_{i, j} \equiv 1 \pmod{p_{i}}\), \(A_{j} \notin \FF_{l_{i, j}}^{\times p_{i}}\), and \(E_{j} \in \FF_{l_{i, j}}^{\times p_{i}}\).
\item \(q_{i, j} \equiv 1 \pmod{p_{i}}\) and \(E_{j} \notin \FF_{q_{i, j}}^{\times p_{i}}\).
\item All the primes \(l_{i, j}\) and \(q_{i, j}\) are mutually distinct, and none of them divides \(2nc_{1}c_{2}C_{1}C_{2}\).
\end{enumerate}
\end{lemma}

\begin{proof}
Fix \(i\) and \(j\), and write \(p = p_{i}\) and \(K = K_{i}\). Put
\[
\widehat{A}_{j} = A_{j}C_{j}^{p} = 4c_{j}^{4}C_{j}^{p - 1}, \qquad
\widehat{E}_{j} = E_{j}b_{j}^{p} = a_{j}b_{j}^{p - 1}.
\]
These are nonzero integers, and
\(\widehat{A}_{j}/A_{j} = C_{j}^{p}\),
\(\widehat{E}_{j}/E_{j} = b_{j}^{p}\). Hence, by Proposition~\ref{prop:kummer-basic}~(i),
\[
K\bigl(\sqrt[p]{A_{j}}\bigr) = K\bigl(\sqrt[p]{\widehat{A}_{j}}\bigr), \qquad
K\bigl(\sqrt[p]{E_{j}}\bigr) = K\bigl(\sqrt[p]{\widehat{E}_{j}}\bigr).
\]
In addition, if a prime \(r\) does not divide \(2c_{j}C_{j}\),
then the \(p\)-th power residuosity of \(A_{j}\) modulo \(r\) coincides with that of \(\widehat{A}_{j}\).
Similarly, if \(r \nmid C_{j}\), then the \(p\)-th power residuosity of \(E_{j}\) modulo \(r\) coincides with that of \(\widehat{E}_{j}\).
Furthermore, since \(A_{j}, E_{j} \notin K^{\times p}\) by Lemma~\ref{lem:AE-independent},
neither \(\widehat{A}_{j}\) nor \(\widehat{E}_{j}\) lies in \(K^{\times p}\).

Let \(\alpha = \sqrt[p]{A_{j}}\) and \(\beta = \sqrt[p]{E_{j}}\), and put
\[
L = K(\alpha, \beta).
\]
By Lemma~\ref{lem:AE-independent} and Corollary~\ref{cor:kummer-independent}~(i), we have
\[
\Gal(L/K) \cong \FF_{p} \oplus \FF_{p}.
\]
Put \(V = \Gal(L/K)\). More precisely, for each \((u, v) \in \FF_{p} \oplus \FF_{p}\), there exists a unique \(\tau_{u, v} \in V\) such that
\[
\tau_{u, v}(\alpha) = \zeta_{p}^{u}\alpha, \qquad
\tau_{u, v}(\beta) = \zeta_{p}^{v}\beta,
\]
and \((u, v) \mapsto \tau_{u, v}\) gives an isomorphism \(\FF_{p} \oplus \FF_{p} \to V\).
Since \(A_{j}, E_{j} \in \QQ^{\times}\) and \(K/\QQ\) is a Galois extension, \(L/\QQ\) is a Galois extension by Corollary~\ref{cor:kummer-independent}~(iii).
Indeed, \(L\) is the splitting field over \(\QQ\) of
\[
(X^{p} - A_{j})(X^{p} - E_{j}).
\]

Put \(G = \Gal(L/\QQ)\). For \(\sigma \in G\), write
\[
\sigma(\zeta_{p}) = \zeta_{p}^{\chi(\sigma)}
\qquad\bigl(\chi(\sigma) \in \FF_{p}^{\times}\bigr).
\]
By Corollary~\ref{cor:kummer-independent}~(iii),
\[
\sigma\tau_{u, v}\sigma^{-1}
= \tau_{\chi(\sigma)u, \chi(\sigma)v}.
\]
The restriction map
\[
G \longrightarrow \Gal(K/\QQ)
\]
is surjective, and \(\Gal(K/\QQ) \cong \FF_{p}^{\times}\). Consequently, each of
\[
\begin{aligned}
\mathcal{C}_{A} &= \{\tau_{u, 0} : u \in \FF_{p}^{\times}\}, \\
\mathcal{C}_{E} &= \{\tau_{0, v} : v \in \FF_{p}^{\times}\}
\end{aligned}
\]
is a single conjugacy class of \(G\).

We apply Chebotarev's density theorem (Corollary~\ref{cor:chebotarev-rational}) to \(\mathcal{C}_{A}\).
There exist infinitely many rational primes \(l\), unramified in \(L/\QQ\), whose Frobenius class is \(\mathcal{C}_{A}\).
Since only finitely many of them divide \(2pc_{j}C_{j}\), we may choose one that moreover satisfies
\[
l \nmid 2pc_{j}C_{j}.
\]
Take such an \(l\) and a prime ideal \(\mathfrak{L}\) of \(L\) lying above \(l\), and write
\[
\Frob{L/\QQ}{\mathfrak{L}} = \tau_{u, 0}
\qquad (u \ne 0).
\]
Since \(\tau_{u, 0} \in V = \Gal(L/K)\), its restriction to \(K\) is the identity map.
Hence \(l\) splits completely in \(K/\QQ\), and
\[
l \equiv 1 \pmod{p}.
\]
Put \(\mathfrak{l} = \mathfrak{L} \cap \mathcal{O}_{K}\).
Since \(l\) splits completely in \(K/\QQ\), the residue degree of \(\mathfrak{l}\) over \(l\) is \(1\), and therefore
\[
\Frob{L/K}{\mathfrak{L}} = \Frob{L/\QQ}{\mathfrak{L}}^{f(\mathfrak{l}/(l))} = \Frob{L/\QQ}{\mathfrak{L}} = \tau_{u, 0}.
\]
Restricting \(\tau_{u, 0}\) to the subextensions \(K(\alpha)/K\) and \(K(\beta)/K\), and using \(u \ne 0\) and Corollary~\ref{cor:kummer-independent}~(ii), we obtain
\[
\begin{aligned}
\Frob{K(\alpha)/K}{\mathfrak{l}} &= \tau_{u, 0}|_{K(\alpha)} \ne \id_{K(\alpha)}, \\
\Frob{K(\beta)/K}{\mathfrak{l}} &= \tau_{u, 0}|_{K(\beta)} = \id_{K(\beta)}.
\end{aligned}
\]
By the choice of \(l\),
\[
l \nmid p, \qquad
l \nmid \widehat{A}_{j}, \qquad
l \nmid \widehat{E}_{j}.
\]
In addition, the residue field of \(\mathfrak{l}\) is \(\FF_{l}\). Applying Proposition~\ref{prop:power-residue-criterion} twice, first with \(F = \QQ\), \(\mathfrak{q} = (l)\), and \(\gamma = \widehat{A}_{j}\), and then with \(\gamma = \widehat{E}_{j}\), we obtain
\[
\widehat{A}_{j} \notin \FF_{l}^{\times p}, \qquad
\widehat{E}_{j} \in \FF_{l}^{\times p}.
\]
Since \(A_{j}\) and \(\widehat{A}_{j}\) have the same \(p\)-th power residuosity modulo \(l\), and likewise for \(E_{j}\) and \(\widehat{E}_{j}\), we conclude that
\[
A_{j} \notin \FF_{l}^{\times p}, \qquad
E_{j} \in \FF_{l}^{\times p}.
\]

Similarly, by applying Corollary~\ref{cor:chebotarev-rational} to \(\mathcal{C}_{E}\) and excluding the finitely many primes dividing \(2pc_{j}C_{j}\), we obtain infinitely many rational primes \(q\) satisfying
\[
q \equiv 1 \pmod{p}, \qquad
q \nmid 2pc_{j}C_{j}, \qquad
E_{j} \notin \FF_{q}^{\times p}.
\]

Since each of the sets above is infinite, we may choose \(l_{i, j}\) and \(q_{i, j}\) successively for the pairs \((i, j)\), excluding at each stage the finitely many primes already chosen and the prime divisors of \(2nc_{1}c_{2}C_{1}C_{2}\).
In this way, (i)--(iii) hold simultaneously.
\end{proof}

In what follows, we refer to the rational primes \(l_{i, j}\) and \(q_{i, j}\) satisfying conditions (i)--(iii) of Lemma~\ref{lem:choice-of-auxiliary-primes} as auxiliary primes.

We fix all the auxiliary primes \(l_{i, j}\) and \(q_{i, j}\) obtained in Lemma~\ref{lem:choice-of-auxiliary-primes}.
Put
\[
M_{0} = \lcm(C_{1}, C_{2}, c_{1}, c_{2}),
\]
and consider the following system of congruence conditions.
\begin{equation}\label{eq:congruence-system}
\left\{
\begin{aligned}
t &\equiv H \pmod{M_{0}}, \\
t &\equiv 0 \pmod{l_{i, j}}\quad (1 \le i \le \omega(n), j = 1, 2), \\
t &\equiv 0 \pmod{q_{i, j}}\quad (1 \le i \le \omega(n), j = 1, 2).
\end{aligned}
\right.
\end{equation}

\begin{lemma}\label{lem:congruence-system-solvable}
There exist infinitely many positive integers \(t\) satisfying the system of congruence conditions~\eqref{eq:congruence-system}.
\end{lemma}

\begin{proof}
By Lemma~\ref{lem:choice-of-auxiliary-primes}, \(M_{0}\) is coprime to all the primes \(l_{i, j}\), \(q_{i, j}\), and the primes \(l_{i, j}\), \(q_{i, j}\) are mutually distinct. Hence the Chinese remainder theorem yields a single residue class, which contains infinitely many positive representatives.
\end{proof}

Let \(\mathcal{T}\) be the set of all positive integers \(t\) satisfying the system of congruence conditions~\eqref{eq:congruence-system} and \(t \ge 2H\).
Note that if \(t \ge 2H\), then \(t^{n} - H^{n} \ge (2^{n} - 1)H^{n} > H^{n} > 0\), so that \((t^{n} - H^{n})^{2} > H^{2n} > C_{j}\).
For \(t \in \mathcal{T}\) and \(j \in \{1, 2\}\), put
\begin{equation}\label{eq:def-W-k}
W = W(t) = t^{n} - H^{n}, \qquad k_{j} = k_{j}(t) = \frac{W}{C_{j}}.
\end{equation}
Furthermore, put
\begin{equation}\label{eq:def-Delta-D}
\Delta_{j}(t) = W^{2} - C_{j} = t^{2n} - 2H^{n}t^{n} + c_{j}^{2}, \qquad
D_{j}(t) = \frac{\Delta_{j}(t)}{16}.
\end{equation}

\begin{lemma}\label{lem:properties-Delta-D}
Let \(t \in \mathcal{T}\). For each \(j \in \{1, 2\}\), the quantities \(k_{j}(t)\) and \(D_{j}(t)\) are integers. Moreover, we have
\begin{equation}\label{eq:D-shift}
D_{1}(t) = D_{2}(t) + m.
\end{equation}
Furthermore, \(W > 0\), \(k_{j}(t) > 0\), \(\Delta_{1}(t) > 0\), and \(\Delta_{2}(t) > 0\).
For any \(r \in \mathcal{R}_{j}\), we also have
\[
v_{r}\bigl(\Delta_{j}(t)\bigr)
= v_{r}\bigl(D_{j}(t)\bigr) = 1.
\]
In particular, \(\Delta_{j}(t)\) and \(D_{j}(t)\) are not squares.
\end{lemma}

\begin{proof}
Since \(t \equiv H \pmod{C_{j}}\), the integer \(C_{j}\) divides \(t^{n} - H^{n} = W\), and thus \(k_{j}(t) \in \ZZ\). By Lemma~\ref{lem:properties-abC}~(i), \(16 \mid C_{j}\), and since \(C_{j} \mid W\), we get \(16 \mid W^{2} - C_{j} = \Delta_{j}(t)\). Hence \(D_{j}(t) \in \ZZ\).

By~\eqref{eq:c-square-difference} and~\eqref{eq:def-abC},
\[
\begin{aligned}
D_{1}(t) - D_{2}(t)
&= \frac{C_{2} - C_{1}}{16}
= \frac{c_{1}^{2} - c_{2}^{2}}{16}
= m.
\end{aligned}
\]
Next, since \(t \ge 2H\), we have \(W = t^{n} - H^{n} > 0\) and \(k_{j}(t) = W/C_{j} > 0\). The inequality \(W^{2} > H^{2n} > C_{j}\) then gives \(\Delta_{j}(t) = W^{2} - C_{j} > 0\).

Fix \(j \in \{1, 2\}\) and \(r \in \mathcal{R}_{j}\).
By~\eqref{eq:def-genus-primes} and Lemma~\ref{lem:properties-abC},
\(v_{r}(C_{j}) = 1\).
From \(W = C_{j}k_{j}(t)\) and \(k_{j}(t) \in \ZZ\), we obtain
\[
\Delta_{j}(t)
= C_{j}\bigl(C_{j}k_{j}(t)^{2} - 1\bigr)
\]
and
\[
C_{j}k_{j}(t)^{2} - 1 \equiv -1 \pmod{r}.
\]
It follows that \(v_{r}(\Delta_{j}(t)) = 1\).
Since \(r\) is an odd prime and \(D_{j}(t) = \Delta_{j}(t)/16\),
we also have \(v_{r}(D_{j}(t)) = 1\).
Since \(\mathcal{R}_{j}\) is nonempty (for instance, \(\lambda_{j, -} \in \mathcal{R}_{j}\)),
neither \(\Delta_{j}(t)\) nor \(D_{j}(t)\) is a square.
\end{proof}

For \(t \in \mathcal{T}\) and \(j \in \{1, 2\}\), put
\begin{equation}\label{eq:def-F}
F_{j}(t) = \QQ\left(\sqrt{D_{j}(t)}\right)
= \QQ\left(\sqrt{\Delta_{j}(t)}\right).
\end{equation}
By Lemma~\ref{lem:properties-Delta-D}, \(D_{j}(t)\) is a positive non-square integer, and therefore \(F_{j}(t)\) is a real quadratic field.

\begin{lemma}\label{lem:genus-primes-ramify}
Let \(t \in \mathcal{T}\) and \(j \in \{1, 2\}\).
Then each prime in \(\mathcal{R}_{j}\) ramifies in \(F_{j}(t)/\QQ\).
\end{lemma}

\begin{proof}
By Lemma~\ref{lem:properties-Delta-D},
\(v_{r}(D_{j}(t)) = 1\) for any \(r \in \mathcal{R}_{j}\).
Hence \(r\) divides the square-free part of \(D_{j}(t)\).
Since \(r\) is an odd prime, \(r\) divides the discriminant of \(F_{j}(t)\), so that \(r\) ramifies in \(F_{j}(t)/\QQ\).
\end{proof}

We denote by \(\varepsilon_{j}(t)\) the fundamental unit of \(F_{j}(t)\). For each \(j \in \{1, 2\}\), we define
\begin{equation}\label{eq:def-eta}
\eta_{j}(t) = \frac{\left(W + \sqrt{\Delta_{j}(t)}\right)^{2}}{C_{j}} \in F_{j}(t)^{\times}.
\end{equation}

\begin{lemma}\label{lem:eta-unit}
Let \(j \in \{1, 2\}\) and \(t \in \mathcal{T}\). Then \(\eta_{j}(t)\) is a unit of infinite order of \(F_{j}(t)\), and it satisfies
\[
\Norm_{F_{j}(t)/\QQ}\bigl(\eta_{j}(t)\bigr) = 1.
\]
Moreover, there exists a unique positive integer \(\nu_{j}(t)\) such that
\begin{equation}\label{eq:eta-fundamental-unit}
\eta_{j}(t) = \varepsilon_{j}(t)^{\nu_{j}(t)}.
\end{equation}
\end{lemma}

\begin{proof}
Since
\[
\Norm_{F_{j}(t)/\QQ}\left(W + \sqrt{\Delta_{j}(t)}\right)
= W^{2} - \Delta_{j}(t) = C_{j},
\]
the norm of \(\eta_{j}(t)\) is \(1\). In addition,
\[
\begin{aligned}
\eta_{j}(t) + \eta_{j}(t)^{-1}
&= \frac{\left(W + \sqrt{\Delta_{j}(t)}\right)^{2} +
\left(W - \sqrt{\Delta_{j}(t)}\right)^{2}}{C_{j}}\\
&= \frac{4W^{2}}{C_{j}} - 2
= 4C_{j}k_{j}(t)^{2} - 2 \in \ZZ.
\end{aligned}
\]
Thus \(\eta_{j}(t)\) is a root of
\[
X^{2} - \bigl(4C_{j}k_{j}(t)^{2} - 2\bigr)X + 1,
\]
and hence it is an algebraic integer. Its inverse is also an algebraic integer, and therefore \(\eta_{j}(t)\) is a unit.

By Lemma~\ref{lem:properties-Delta-D}, \(W^{2} > C_{j}\) and \(W > 0\), and hence
\[
\eta_{j}(t) > \frac{W^{2}}{C_{j}} > 1.
\]
Consequently, \(\eta_{j}(t)\) has infinite order. Since \(\mathcal{O}_{F_{j}(t)}^{\times} = \{\pm 1\} \times \langle \varepsilon_{j}(t) \rangle\), we can write \(\eta_{j}(t) = \pm\varepsilon_{j}(t)^{\nu}\) with \(\nu \in \ZZ\). Since \(\eta_{j}(t) > 1\) and \(\varepsilon_{j}(t) > 1\), the sign is \(+\) and \(\nu > 0\). Furthermore, since \(\varepsilon_{j}(t) > 1\), the map \(\nu \mapsto \varepsilon_{j}(t)^{\nu}\) is injective, so that this \(\nu\) is unique. If we put \(\nu_{j}(t) = \nu\), then we obtain~\eqref{eq:eta-fundamental-unit}.
\end{proof}

\section{Divisibility of class numbers}\label{sec:divisibility}

In this section, we show that the class numbers of the real quadratic fields \(F_{j}(t)\) constructed in Section~\ref{sec:construction} are divisible by \(2^{e}n\).
Genus theory handles the \(2\)-part of the class number (Lemma~\ref{lem:genus-two-part}),
while the class group criterion handles the odd part \(n\) (Theorem~\ref{thm:class-group-criterion}); the two arguments are independent.
We first treat the \(2\)-part.

\begin{lemma}\label{lem:two-part-divisibility}
Let \(t \in \mathcal{T}\) and \(j \in \{1, 2\}\). Then
\[
2^{e} \mid h\bigl(F_{j}(t)\bigr).
\]
\end{lemma}

\begin{proof}
By~\eqref{eq:def-genus-primes}, \(|\mathcal{R}_{j}| = e + 2\).
By Lemma~\ref{lem:genus-primes-ramify}, each prime in \(\mathcal{R}_{j}\) ramifies in
\(F_{j}(t)/\QQ\), and hence
\[
\omega\bigl(d_{F_{j}(t)}\bigr) \ge e + 2.
\]
Since \(F_{j}(t)\) is a real quadratic field,
the assertion follows from Lemma~\ref{lem:genus-two-part}~\textrm{(ii)}.
\end{proof}

In what follows, let \(t \in \mathcal{T}\). For each \(j \in \{1, 2\}\), put
\begin{equation}\label{eq:def-QUS}
Q_{j} = \frac{C_{j}}{4c_{j}^{2}}, \qquad
U_{j}(t) = \frac{c_{j}^{2} - H^{n}t^{n}}{c_{j}}, \qquad
S_{j}(t) = \frac{U_{j}(t)^{2}}{Q_{j}}.
\end{equation}

\begin{lemma}\label{lem:S-identities}
Let \(t \in \mathcal{T}\). For each \(j \in \{1, 2\}\), we have
\begin{align}
\Delta_{j}(t) &= U_{j}(t)^{2} - 4Q_{j}t^{2n}, \label{eq:Delta-via-U}\\
S_{j}(t) &= 4C_{j}\bigl(1 + H^{n}k_{j}(t)\bigr)^{2} \in \ZZ, \label{eq:S-integral}\\
S_{j}(t)^{2} - 4S_{j}(t)t^{2n}
&= 16c_{j}^{2}\bigl(1 + H^{n}k_{j}(t)\bigr)^{2}\Delta_{j}(t). \label{eq:S-discriminant}
\end{align}
Consequently,
\begin{equation}\label{eq:S-field}
\QQ\left(\sqrt{S_{j}(t)^{2} - 4S_{j}(t)t^{2n}}\right) = F_{j}(t).
\end{equation}
\end{lemma}

\begin{proof}
Equation~\eqref{eq:Delta-via-U} follows by direct computation. Furthermore, \(t^{n} = H^{n} + C_{j}k_{j}(t)\) gives
\[
c_{j}^{2} - H^{n}t^{n} = -C_{j}\bigl(1 + H^{n}k_{j}(t)\bigr).
\]
Therefore,
\[
S_{j}(t) = \frac{4(c_{j}^{2} - H^{n}t^{n})^{2}}{C_{j}}
= 4C_{j}\bigl(1 + H^{n}k_{j}(t)\bigr)^{2} \in \ZZ.
\]
Furthermore,
\[
\begin{aligned}
S_{j}(t)^{2} - 4S_{j}(t)t^{2n}
&= \left(\frac{U_{j}(t)}{Q_{j}}\right)^{2}
\left(U_{j}(t)^{2} - 4Q_{j}t^{2n}\right)\\
&= \left(\frac{4c_{j}(c_{j}^{2} - H^{n}t^{n})}{C_{j}}\right)^{2}\Delta_{j}(t)\\
&= 16c_{j}^{2}\bigl(1 + H^{n}k_{j}(t)\bigr)^{2}\Delta_{j}(t).
\end{aligned}
\]
By Lemma~\ref{lem:properties-Delta-D}, \(k_{j}(t) > 0\), so that
\[
4c_{j}\bigl(1 + H^{n}k_{j}(t)\bigr) \ne 0.
\]
Thus the factor multiplying \(\Delta_{j}(t)\) in~\eqref{eq:S-discriminant} is the square of a nonzero rational number,
and~\eqref{eq:S-field} follows.
\end{proof}

\begin{lemma}\label{lem:S-coprime-t}
Let \(t \in \mathcal{T}\). For each \(j \in \{1, 2\}\), we have
\[
\gcd\bigl(S_{j}(t), t^{2}\bigr) = 1.
\]
\end{lemma}

\begin{proof}
Since \(M_{0}\) is a multiple of both \(C_{j}\) and \(c_{j}\), the system of congruence conditions~\eqref{eq:congruence-system} gives
\[
t \equiv H \pmod{C_{j}}, \qquad t \equiv H \pmod{c_{j}}.
\]
By Lemma~\ref{lem:properties-abC}~(ii), \(\gcd(H, C_{j}) = \gcd(H, c_{j}) = 1\), and hence
\[
\gcd(t, C_{j}) = \gcd(t, c_{j}) = 1;
\]
consequently \(\gcd(t, C_{j}c_{j}) = 1\).
Take a prime \(r\) dividing \(t\). Since \(C_{j}k_{j}(t) = t^{n} - H^{n}\), we find
\[
\begin{aligned}
C_{j}\bigl(1 + H^{n}k_{j}(t)\bigr)
&\equiv C_{j} - H^{2n}\\
&\equiv -c_{j}^{2} \not\equiv 0 \pmod{r}.
\end{aligned}
\]
In addition, by Lemma~\ref{lem:properties-abC}~(i), \(16 \mid C_{j} \mid M_{0}\), and since \(t \equiv H \pmod{M_{0}}\) and \(H\) is odd, \(t\) is also odd.
Hence \(r \ne 2\). From~\eqref{eq:S-integral} we obtain \(r \nmid S_{j}(t)\), and the assertion follows.
\end{proof}

\begin{lemma}\label{lem:reduction-at-auxiliary-primes}
Let \(t \in \mathcal{T}\), \(i \in \{1, 2, \dots, \omega(n)\}\), and \(j \in \{1, 2\}\). Then the following hold.
\begin{enumerate}[label=\textrm{(\roman*)}]
\item The primes \(l_{i, j}\) and \(q_{i, j}\) split completely in \(F_{j}(t)/\QQ\).
\item Let \(r \in \{l_{i, j}, q_{i, j}\}\), and let \(\mathfrak{r}\) be a prime ideal of \(F_{j}(t)\) lying above \(r\). Then
\[
\eta_{j}(t) \equiv E_{j}\quad\text{or}\quad E_{j}^{-1} \pmod{\mathfrak{r}}.
\]
\item \(S_{j}(t) \equiv A_{j} \pmod{l_{i, j}}\).
\end{enumerate}
\end{lemma}

\begin{proof}
Let \(r \in \{l_{i, j}, q_{i, j}\}\). Since \(t \equiv 0 \pmod{r}\) by the system of congruence conditions~\eqref{eq:congruence-system}, we have
\[
\Delta_{j}(t) \equiv c_{j}^{2} \pmod{r}.
\]
Since \(r \nmid 2c_{j}\) by Lemma~\ref{lem:choice-of-auxiliary-primes}~(iii), the prime \(r\) splits completely in \(F_{j}(t)/\QQ\), which proves (i).

Since \(r \nmid C_{j}\) by Lemma~\ref{lem:choice-of-auxiliary-primes}~(iii), the integer \(C_{j}\) is invertible modulo \(\mathfrak{r}\).
Also, \(t \equiv 0 \pmod{r}\) gives
\[
W = t^{n} - H^{n} \equiv -H^{n} \pmod{r}.
\]
Since \(r\) splits completely in \(F_{j}(t)\), the residue field \(\mathcal{O}_{F_{j}(t)}/\mathfrak{r}\) is isomorphic to \(\FF_{r}\). From \(\bigl(\sqrt{\Delta_{j}(t)}\bigr)^{2} = \Delta_{j}(t) \equiv c_{j}^{2} \pmod{\mathfrak{r}}\), we therefore obtain \(\sqrt{\Delta_{j}(t)} \equiv \pm c_{j} \pmod{\mathfrak{r}}\).
If \(\sqrt{\Delta_{j}(t)} \equiv c_{j} \pmod{\mathfrak{r}}\), then reducing \(C_{j}\eta_{j}(t) = (W + \sqrt{\Delta_{j}(t)})^{2}\) modulo \(\mathfrak{r}\) gives
\[
\begin{aligned}
\eta_{j}(t)
&\equiv \frac{(-H^{n} + c_{j})^{2}}{C_{j}}
= \frac{(H^{n} - c_{j})^{2}}{(H^{n} - c_{j})(H^{n} + c_{j})}\\
&= \frac{H^{n} - c_{j}}{H^{n} + c_{j}} = E_{j} \pmod{\mathfrak{r}}.
\end{aligned}
\]
If \(\sqrt{\Delta_{j}(t)} \equiv -c_{j} \pmod{\mathfrak{r}}\), then a similar computation gives \(\eta_{j}(t) \equiv b_{j}^{2}/C_{j} = E_{j}^{-1} \pmod{\mathfrak{r}}\).

Finally, from the identity \(S_{j}(t) = 4(c_{j}^{2} - H^{n}t^{n})^{2}/C_{j}\) shown in the proof of Lemma~\ref{lem:S-identities}, together with \(t \equiv 0 \pmod{l_{i, j}}\) and \(l_{i, j} \nmid C_{j}\), we obtain
\[
S_{j}(t) = \frac{4(c_{j}^{2} - H^{n}t^{n})^{2}}{C_{j}} \equiv \frac{4c_{j}^{4}}{C_{j}} = A_{j} \pmod{l_{i, j}}.
\]
\end{proof}

\begin{lemma}\label{lem:exponent-and-unit-residue}
Let \(t \in \mathcal{T}\). For each \(i \in \{1, 2, \dots, \omega(n)\}\) and \(j \in \{1, 2\}\), we have
\[
p_{i} \nmid \nu_{j}(t).
\]
Moreover, for any prime ideal \(\mathfrak{l}\) of \(F_{j}(t)\) lying above \(l_{i, j}\), we have
\[
\varepsilon_{j}(t) \in (\mathcal{O}_{F_{j}(t)}/\mathfrak{l})^{\times p_{i}}.
\]
\end{lemma}

\begin{proof}
By Lemma~\ref{lem:reduction-at-auxiliary-primes}, the residue class of \(\eta_{j}(t)\) modulo any prime ideal lying above \(q_{i, j}\) is \(E_{j}\) or \(E_{j}^{-1}\). By the choice of \(q_{i, j}\), \(E_{j}\) is a \(p_{i}\)-th power non-residue modulo \(q_{i, j}\), and hence so is \(E_{j}^{-1}\). Thus \(\eta_{j}(t)\) is a \(p_{i}\)-th power non-residue modulo every prime ideal lying above \(q_{i, j}\). If \(p_{i} \mid \nu_{j}(t)\), then by~\eqref{eq:eta-fundamental-unit} \(\eta_{j}(t)\) is a \(p_{i}\)-th power in \(F_{j}(t)\), and therefore it is a \(p_{i}\)-th power residue modulo every prime ideal lying above \(q_{i, j}\), a contradiction. Consequently, \(p_{i} \nmid \nu_{j}(t)\).

On the other hand, by Lemma~\ref{lem:reduction-at-auxiliary-primes} and the choice of \(l_{i, j}\), the unit \(\eta_{j}(t)\) is a \(p_{i}\)-th power residue modulo every prime ideal lying above \(l_{i, j}\). Applying Lemma~\ref{lem:unit-residuosity-transfer} with \(\mathfrak{l}\) a prime ideal above \(l_{i, j}\) and \(\mathfrak{q}\) a prime ideal above \(q_{i, j}\), we conclude that \(\varepsilon_{j}(t)\) is also a \(p_{i}\)-th power residue modulo every prime ideal lying above \(l_{i, j}\).
\end{proof}

\begin{lemma}\label{lem:criterion-hypotheses}
Let \(t \in \mathcal{T}\). Then, for each \(i \in \{1, 2, \dots, \omega(n)\}\) and \(j \in \{1, 2\}\), the following hold.
\begin{enumerate}[label=\textrm{(\roman*)}]
\item \(\gcd(S_{j}(t), t^{2}) = 1\).
\item The integer \(S_{j}(t)\) is a \(p_{i}\)-th power non-residue modulo \(l_{i, j}\).
\item The prime \(l_{i, j}\) divides \(t^{2}\).
\item The fundamental unit \(\varepsilon_{j}(t)\) of \(F_{j}(t)\) is a \(p_{i}\)-th power residue modulo every prime ideal of \(F_{j}(t)\) lying above \(l_{i, j}\).
\end{enumerate}
\end{lemma}

\begin{proof}
Assertion (i) follows from Lemma~\ref{lem:S-coprime-t}.

By Lemma~\ref{lem:reduction-at-auxiliary-primes}, \(S_{j}(t) \equiv A_{j} \pmod{l_{i, j}}\), and by Lemma~\ref{lem:choice-of-auxiliary-primes}~(i), \(A_{j}\) is a \(p_{i}\)-th power non-residue modulo \(l_{i, j}\). Hence (ii) holds.

Assertion (iii) follows from the congruence condition \(t \equiv 0 \pmod{l_{i, j}}\).

Assertion (iv) is precisely Lemma~\ref{lem:exponent-and-unit-residue}. The primes \(l_{i, j}\) and \(q_{i, j}\) used here are fixed in Lemma~\ref{lem:choice-of-auxiliary-primes} before \(t\) is chosen. Thus there is no circular dependence between \(t\) and the choice of the auxiliary primes.
\end{proof}

With these preparations, we apply Theorem~\ref{thm:class-group-criterion} with
\(X = S_{j}(t)\) and \(Z = t^{2}\) to obtain the main result of this section.

\begin{lemma}\label{lem:both-class-groups}
Let \(t \in \mathcal{T}\). Then the class groups of \(F_{1}(t)\) and \(F_{2}(t)\) both contain an element of order \(n\),
and moreover
\[
2^{e}n \mid h\bigl(F_{j}(t)\bigr)
\qquad (j = 1, 2).
\]
\end{lemma}

\begin{proof}
Fix \(j \in \{1, 2\}\), and put
\[
\Delta_{j}^{\ast}(t) = S_{j}(t)^{2} - 4S_{j}(t)t^{2n}.
\]
By~\eqref{eq:S-discriminant},
\[
\Delta_{j}^{\ast}(t)
= \bigl(4c_{j}(1 + H^{n}k_{j}(t))\bigr)^{2}\Delta_{j}(t).
\]
By the proof of Lemma~\ref{lem:S-identities},
\(4c_{j}(1 + H^{n}k_{j}(t)) \ne 0\), and
by Lemma~\ref{lem:properties-Delta-D}, \(\Delta_{j}(t)\) is a positive non-square integer.
Since the square factor in the above identity is nonzero,
\(\Delta_{j}^{\ast}(t)\) is also a positive non-square integer, and
\[
\QQ\left(\sqrt{\Delta_{j}^{\ast}(t)}\right) = F_{j}(t).
\]

We apply Theorem~\ref{thm:class-group-criterion} with
\[
(X, Z, D) = \bigl(S_{j}(t), t^{2}, \Delta_{j}^{\ast}(t)\bigr).
\]
Then we have
\[
X^{2} - D = 4XZ^{n}.
\]
For each \(i\), take \(l_{i} = l_{i, j}\). By Lemma~\ref{lem:choice-of-auxiliary-primes}~(i), \(l_{i, j} \equiv 1 \pmod{p_{i}}\), so (C1) holds.
Assertions (ii), (iii), and (iv) of Lemma~\ref{lem:criterion-hypotheses} give (C2), (C3), and (C4), respectively, while assertion (i) of the same lemma gives \(\gcd(X, Z) = 1\).
Hence the class group of \(F_{j}(t)\) contains an element of order \(n\), and in particular
\[
n \mid h\bigl(F_{j}(t)\bigr).
\]
On the other hand, by Lemma~\ref{lem:two-part-divisibility},
\[
2^{e} \mid h\bigl(F_{j}(t)\bigr).
\]
Since \(n\) is odd, \(\gcd(2^{e}, n) = 1\), and hence
\[
2^{e}n \mid h\bigl(F_{j}(t)\bigr).
\]
Applying this argument to \(j = 1\) and \(j = 2\) proves the assertion.
\end{proof}

\section{Infinitude of pairs of quadratic fields}\label{sec:infinitude}

In this section, we show that the family of quadratic fields constructed in Section~\ref{sec:construction} contains infinitely many distinct fields.
The main tool is Lemma~\ref{lem:siegel-quadratic-fields}, a consequence of Siegel's theorem.

For each \(j \in \{1, 2\}\), put
\begin{equation}\label{eq:def-f}
f_{j}(X) = X^{2n} - 2H^{n}X^{n} + c_{j}^{2}.
\end{equation}
Then \(f_{j}(t) = \Delta_{j}(t)\).

\begin{lemma}\label{lem:f-separable}
Neither \(f_{1}(X)\) nor \(f_{2}(X)\) has a multiple root.
\end{lemma}

\begin{proof}
The derivative of \(f_{j}\) is
\[
f_{j}'(X) = 2nX^{n - 1}(X^{n} - H^{n}).
\]
If \(f_{j}\) and \(f_{j}'\) have a common root, then that root satisfies \(X = 0\) or \(X^{n} = H^{n}\). However,
\[
f_{j}(0) = c_{j}^{2} \ne 0,
\]
and when \(X^{n} = H^{n}\),
\[
f_{j}(X) = c_{j}^{2} - H^{2n} = -C_{j} \ne 0.
\]
Thus there is no common root.
\end{proof}

\begin{lemma}\label{lem:infinitely-many-fields}
The set
\[
\left\{F_{2}(t) : t \in \mathcal{T}\right\}
\]
contains infinitely many distinct real quadratic fields.
\end{lemma}

\begin{proof}
By Lemma~\ref{lem:congruence-system-solvable}, \(\mathcal{T}\) is infinite.
In addition, by Lemma~\ref{lem:f-separable},
\(f_{2}\) is a polynomial of degree \(2n \ge 6\) with no multiple roots.
Consequently, applying Lemma~\ref{lem:siegel-quadratic-fields} to
\(f_{2}\) and \(\mathcal{T}\) shows that
\[
\left\{\QQ\left(\sqrt{f_{2}(t)}\right) : t \in \mathcal{T}\right\}
\]
contains infinitely many distinct quadratic fields.
Since \(f_{2}(t) = \Delta_{2}(t)\), equation~\eqref{eq:def-F} shows that this set is exactly \(\{F_{2}(t) : t \in \mathcal{T}\}\).
By Lemma~\ref{lem:properties-Delta-D}, all of these fields are real quadratic fields.
\end{proof}

\begin{remark}\label{rem:distinct-pairs}
For any \(t \in \mathcal{T}\), the relation \(D_{1}(t) = D_{2}(t) + m\) holds. Hence distinct fields \(F_{2}(t)\) give rise to distinct pairs of the form
\[
\left(F_{2}(t), F_{1}(t)\right)
= \left(\QQ(\sqrt{D_{2}(t)}), \QQ(\sqrt{D_{2}(t) + m})\right).
\]

Moreover, the two fields in each pair are distinct.
Indeed, put \(\lambda = \lambda_{1, -}\); then
\(v_{\lambda}(\Delta_{1}(t)) = 1\) by Lemma~\ref{lem:properties-Delta-D}.
On the other hand, since \(\Delta_{1}(t) - \Delta_{2}(t) = 16m\) and \(\lambda \nmid 16m\) by Lemma~\ref{lem:choice-of-H}~(iv),
it follows that \(v_{\lambda}(\Delta_{2}(t)) = 0\).
Since \(\lambda\) is an odd prime, it ramifies in \(F_{1}(t)\)
but not in \(F_{2}(t)\).
Therefore \(F_{1}(t) \ne F_{2}(t)\).
\end{remark}

\section{Proof of Theorem~\ref{thm:main}}\label{sec:proof-of-main-theorem}

First, we treat the case in which \(N\) is a power of \(2\), using genus theory alone.

\begin{proposition}\label{prop:pure-two-power-pairs}
For any integers \(e \ge 1\) and \(m \ge 1\),
there exist infinitely many distinct pairs of real quadratic fields
\[
\left(\QQ(\sqrt{D}), \QQ(\sqrt{D + m})\right),
\qquad D \in \ZZ, \quad D > 0,
\]
such that both class numbers are divisible by \(2^{e}\).
\end{proposition}

\begin{proof}
Take an arbitrary finite set \(\mathcal{F}\) of quadratic fields.
Let \(\mathcal{S}\) be the set of all rational primes that ramify in at least one of the fields in \(\mathcal{F}\).
The set \(\mathcal{S}\) is finite.
Choose \(2(e + 2)\) mutually distinct odd primes
\[
r_{0, k},\ r_{1, k} \qquad (1 \le k \le e + 2)
\]
such that
\[
r_{j, k} > m, \qquad r_{j, k} \notin \mathcal{S}.
\]
By the Chinese remainder theorem, the integers \(D\) satisfying
\begin{equation}\label{eq:pure-two-crt}
\left\{
\begin{aligned}
D &\equiv r_{0, k} \pmod{r_{0, k}^{2}} &&(1 \le k \le e + 2),\\
D &\equiv r_{1, k} - m \pmod{r_{1, k}^{2}} &&(1 \le k \le e + 2)
\end{aligned}
\right.
\end{equation}
form a single residue class, which contains infinitely many positive integers.
Fix one such positive integer \(D\).

For each \(k\), we have
\[
v_{r_{0, k}}(D) = 1, \qquad v_{r_{1, k}}(D + m) = 1.
\]
Thus \(D\) and \(D + m\) are not squares, and
\[
L_{0} = \QQ(\sqrt{D}), \qquad L_{1} = \QQ(\sqrt{D + m})
\]
are real quadratic fields.
Furthermore, the primes \(r_{0, 1}, \dots, r_{0, e + 2}\) all ramify in \(L_{0}\), and
the primes \(r_{1, 1}, \dots, r_{1, e + 2}\) all ramify in \(L_{1}\).
Hence
\[
\omega(d_{L_{0}}) \ge e + 2, \qquad
\omega(d_{L_{1}}) \ge e + 2.
\]
By Lemma~\ref{lem:genus-two-part}~\textrm{(ii)},
\[
2^{e} \mid h(L_{0}), \qquad 2^{e} \mid h(L_{1}).
\]

In addition, \(D \equiv 0 \pmod{r_{0, 1}}\), whereas
\[
D + m \equiv m \not\equiv 0 \pmod{r_{0, 1}},
\]
where we have used \(r_{0, 1} > m\).
Therefore \(r_{0, 1}\) ramifies in \(L_{0}\) but not in \(L_{1}\), and
\(L_{0} \ne L_{1}\).
Furthermore, since \(r_{0, 1}, r_{1, 1} \notin \mathcal{S}\),
neither \(L_{0}\) nor \(L_{1}\) belongs to \(\mathcal{F}\).

Thus, given any finite set \(\mathcal{F}\) of quadratic fields, we can construct a pair with the required properties whose two fields lie outside \(\mathcal{F}\).
Repeating the construction after adding all previously obtained fields to \(\mathcal{F}\) yields infinitely many distinct pairs with the required properties.
\end{proof}

\begin{proof}[Proof of Theorem~\ref{thm:main}]
Let \(N \ge 2\) and \(m \ge 1\) be integers, and write
\[
N = 2^{e}n, \qquad e \ge 0, \qquad n\ \text{odd}.
\]

First, suppose that \(n = 1\). Then \(e \ge 1\), and
the assertion follows by applying Proposition~\ref{prop:pure-two-power-pairs}.

In what follows, suppose that \(n \ge 3\).
We carry out the construction of Sections~\ref{sec:construction}--\ref{sec:infinitude}
for these \(e\), \(n\), and \(m\).
First, we choose \(H\), \(\lambda_{j, \pm}\), and \(\xi_{j, k}\) as in Lemma~\ref{lem:choice-of-H},
and define \(C_{j}\), \(A_{j}\), \(E_{j}\) by~\eqref{eq:def-abC} and~\eqref{eq:def-AE}.
Here the \(e + 2\) primes in \(\mathcal{R}_{j}\) divide \(C_{j}\) exactly once and eventually ramify in \(F_{j}(t)\); these primes are responsible for the \(2^{e}\)-divisibility via genus theory.
Next, by Lemma~\ref{lem:choice-of-auxiliary-primes},
we choose all the auxiliary primes \(l_{i, j}\), \(q_{i, j}\) for the odd part \(n\).
Since these primes are chosen so as not to divide \(C_{1}C_{2}\),
they do not coincide with any of the primes in \(\mathcal{R}_{1} \cup \mathcal{R}_{2}\) used for genus theory.

We then let \(\mathcal{T}\) be the set of all positive integers \(t\) satisfying the system of congruence conditions~\eqref{eq:congruence-system} and \(t \ge 2H\).
Since \(M_{0}\) is a multiple of \(C_{1}\) and \(C_{2}\),
the condition \(t \equiv H \pmod{M_{0}}\) preserves the ramification conditions needed for genus theory.
On the other hand, the congruences \(t \equiv 0 \pmod{l_{i, j}q_{i, j}}\) transfer the residue conditions on \(A_{j}\) and \(E_{j}\) at the auxiliary primes \(l_{i, j}\), \(q_{i, j}\) to \(S_{j}(t)\) and \(\eta_{j}(t)\), thereby establishing the conditions needed for the odd part \(n\).
Lemma~\ref{lem:congruence-system-solvable}, which is an application of the Chinese remainder theorem, shows that these congruence conditions can be satisfied simultaneously.
The choices are made in the order
\[
(H, \{\lambda_{j, \pm}, \xi_{j, k}\})
\longrightarrow (A_{j}, E_{j})
\longrightarrow \{l_{i, j}, q_{i, j}\}_{i, j}
\longrightarrow t
\longrightarrow F_{j}(t), \eta_{j}(t), \varepsilon_{j}(t),
\]
and there is no circular dependence.

For any \(t \in \mathcal{T}\), put
\[
D = D_{2}(t).
\]
Then, by Lemma~\ref{lem:properties-Delta-D},
\[
\begin{aligned}
F_{2}(t) &= \QQ(\sqrt{D}), \\
F_{1}(t) &= \QQ(\sqrt{D + m})
\end{aligned}
\]
are both real quadratic fields.
By Lemma~\ref{lem:both-class-groups}, the class group of each field contains an element of order \(n\),
and
\[
N = 2^{e}n \mid h\bigl(F_{j}(t)\bigr)
\qquad (j = 1, 2).
\]
Furthermore, by Lemma~\ref{lem:infinitely-many-fields} and Remark~\ref{rem:distinct-pairs},
there are infinitely many distinct pairs of this form.
This completes the proof of Theorem~\ref{thm:main}.
\end{proof}

\section{The imaginary quadratic case}\label{sec:imaginary-quadratic-case}

In this section, as a supplementary application of the construction used in the real quadratic case,
we give another proof of the following known result of Xie and Chao.

\begin{theorem}[{Xie--Chao \cite[Theorem~1.2]{XieChao2020}}]\label{thm:xie-chao-imaginary}
For any odd integer \(n \ge 3\) and any integer \(m \ge 1\),
there exist infinitely many distinct pairs of imaginary quadratic fields
\[
\left(\QQ(\sqrt{D}), \QQ(\sqrt{D + m})\right),
\qquad D \in \ZZ, \quad D + m < 0,
\]
such that the class group of each field contains an element of order \(n\).
\end{theorem}

The key difference from the real quadratic case is that every imaginary quadratic field
other than \(\QQ(\sqrt{-1})\) and \(\QQ(\sqrt{-3})\) has unit group \(\{\pm 1\}\).
Thus the analogue of condition \textrm{(C4)} of Theorem~\ref{thm:class-group-criterion}, which concerns the fundamental unit in the real quadratic case, is unnecessary here, and we may use the following criterion.

\begin{lemma}[Class group criterion for imaginary quadratic fields]\label{lem:imaginary-class-group-criterion}
Let \(D < 0\) be an integer, and put \(F = \QQ(\sqrt{D})\).
Assume that \(F\) is neither \(\QQ(\sqrt{-1})\) nor \(\QQ(\sqrt{-3})\).
Let \(n \ge 3\) be an odd integer, and let \(p_{1}, \dots, p_{\omega(n)}\) be
all the distinct prime factors of \(n\).
Suppose that integers \(X, Z\) satisfy
\[
X^{2} - D = 4XZ^{n}, \qquad \gcd(X, Z) = 1,
\]
and that, for each \(i \in \{1, \dots, \omega(n)\}\), there exists a prime \(l_{i}\) satisfying
\[
l_{i} \equiv 1 \pmod{p_{i}}, \qquad
X \notin \FF_{l_{i}}^{\times p_{i}}, \qquad l_{i} \mid Z.
\]
Then the class group of \(F\) contains an element of order \(n\).
\end{lemma}

\begin{proof}
Let \(\sigma\) be the nontrivial automorphism of \(F/\QQ\), and put
\[
\alpha = \frac{X - 2Z^{n} + \sqrt{D}}{2}.
\]
By the assumed equation,
\[
\alpha + \alpha^{\sigma} = X - 2Z^{n}, \qquad
\alpha\alpha^{\sigma} = Z^{2n}.
\]
Therefore \(\alpha\) is a root of the monic polynomial
\(T^{2} - (X - 2Z^{n})T + Z^{2n}\) with integer coefficients, and
\(\alpha \in \mathcal{O}_{F}\).
Also, if \(Z = 0\), then \(D = X^{2}\), contrary to \(D < 0\); hence
\(Z \ne 0\) and \(\alpha \ne 0\).

Suppose that a prime ideal divides both \((\alpha)\) and \((\alpha^{\sigma})\). Then the rational prime below it divides both \(Z^{2n}\) and \(X - 2Z^{n}\), contrary to \(\gcd(X, Z) = 1\). Hence the two ideals are coprime.
Since \((\alpha)(\alpha^{\sigma}) = (Z)^{2n}\), unique factorization of ideals shows that
there exists an integral ideal \(\mathfrak{a}\) such that
\[
(\alpha) = \mathfrak{a}^{2n}.
\]

Fix \(i\). From \(l_{i} \mid Z\) and \(\gcd(X, Z) = 1\), we see that
\(l_{i} \nmid X\), and
\[
D \equiv X^{2} \not\equiv 0 \pmod{l_{i}}.
\]
Since \(p_{i}\) is an odd prime and \(l_{i} \equiv 1 \pmod{p_{i}}\),
the prime \(l_{i}\) is odd, and it splits completely in \(F/\QQ\).
If we choose a prime ideal \(\mathfrak{l}_{i}\) lying above \(l_{i}\) such that
\(\sqrt{D} \equiv X \pmod{\mathfrak{l}_{i}}\), then
\[
\alpha \equiv X \pmod{\mathfrak{l}_{i}}, \qquad
\mathcal{O}_{F}/\mathfrak{l}_{i} \cong \FF_{l_{i}}.
\]
It follows that \(\alpha\) is a \(p_{i}\)-th power non-residue modulo \(\mathfrak{l}_{i}\).

If \(\mathfrak{a}^{2n/p_{i}}\) is principal, then
there exists \(\beta \in \mathcal{O}_{F} \setminus \{0\}\) with
\(\mathfrak{a}^{2n/p_{i}} = (\beta)\).
Since \(\mathcal{O}_{F}^{\times} = \{\pm 1\}\), we have
\[
\alpha = \pm\beta^{p_{i}}.
\]
Since the residue class of \(\alpha\) is nonzero, the residue class of \(\beta\) is also nonzero,
and since \(p_{i}\) is odd, the right-hand side is a \(p_{i}\)-th power residue modulo \(\mathfrak{l}_{i}\).
This is a contradiction.
Hence, for all \(i\),
\[
[\mathfrak{a}]^{2n/p_{i}} \ne 1.
\]
Now put \(g = [\mathfrak{a}]^{2}\). Then
\(g^{n} = 1\) and \(g^{n/p_{i}} \ne 1\) for all \(i\).
If the order of \(g\) were a proper divisor of \(n\),
then it would divide \(n/p_{i}\) for some \(i\), a contradiction.
Hence the order of \(g\) is exactly \(n\).
\end{proof}

With this criterion, the explicit unit \(\eta_{j}(t)\) and
the auxiliary primes \(q_{i, j}\), which were used in the real quadratic case to control the \(p_{i}\)-th power residuosity of the explicit unit, are no longer needed.
In what follows, we modify the construction so that the radicands become negative,
and we give another proof using this criterion and Lemma~\ref{lem:siegel-quadratic-fields}.

\begin{proof}[Another proof of Theorem~\ref{thm:xie-chao-imaginary}]
Let \(n \ge 3\) be an odd integer and let \(m \ge 1\) be an integer.
Let \(p_{1}, \dots, p_{\omega(n)}\) be all the distinct prime factors of \(n\),
and put \(K_{i} = \QQ(\zeta_{p_{i}})\).
We use \(c_{1} = 4m + 1\) and \(c_{2} = 4m - 1\) from~\eqref{eq:def-c}.

First, we choose an integer \(B\) and primes \(\rho_{1}, \rho_{2}\).
For each \(j \in \{1, 2\}\), put
\[
g_{j}(X) = X^{2n} + c_{j}^{2}.
\]
By Schur's theorem, there exist infinitely many primes that divide some integer value of \(g_{j}\).
Thus we can choose mutually distinct primes
\(\rho_{1}, \rho_{2}\) not dividing \(6nmc_{1}c_{2}\) and integers \(x_{1}, x_{2}\) such that
\[
g_{j}(x_{j}) \equiv 0 \pmod{\rho_{j}}.
\]
Since \(\rho_{j} \nmid c_{j}\), we find that \(\rho_{j} \nmid x_{j}\), and
\[
g'_{j}(x_{j}) = 2n x_{j}^{2n - 1} \not\equiv 0 \pmod{\rho_{j}}.
\]
Hence, among the \(\rho_{j}\) residue classes modulo \(\rho_{j}^{2}\) lifting the residue class of \(x_{j}\) modulo \(\rho_{j}\),
exactly one is a root of \(g_{j}\) modulo \(\rho_{j}^{2}\).
If we choose a lift \(y_{j}\) different from this one, then
\[
v_{\rho_{j}}\bigl(g_{j}(y_{j})\bigr) = 1.
\]

Since the moduli \(2\), \(\rad(c_{1}c_{2})\), \(\rho_{1}^{2}\), and \(\rho_{2}^{2}\) are pairwise coprime,
the Chinese remainder theorem yields an integer \(B\) satisfying
\[
\begin{aligned}
B &\equiv 0 \pmod{2}, \\
B &\equiv 1 \pmod{\rad(c_{1}c_{2})}, \\
B &\equiv y_{j} \pmod{\rho_{j}^{2}} \qquad (j = 1, 2).
\end{aligned}
\]
Replacing \(B\) by a sufficiently large positive representative of this residue class, we may assume that \(B > c_{1}\).
Then \(B\) is even, and
\[
\gcd(B, c_{1}c_{2}) = 1, \qquad
v_{\rho_{j}}(B^{2n} + c_{j}^{2}) = 1.
\]
In what follows, put
\[
\widetilde{C}_{j} = B^{2n} + c_{j}^{2}, \qquad
\widetilde{A}_{j} = \frac{4c_{j}^{4}}{\widetilde{C}_{j}}, \qquad
\widetilde{M} = \lcm(\widetilde{C}_{1}, \widetilde{C}_{2}, c_{1}, c_{2}).
\]
Then \(\widetilde{C}_{1}, \widetilde{C}_{2}\) are positive odd integers,
and \(\widetilde{M}\) is also odd.
We also have \(\gcd(B, \widetilde{M}) = 1\), because
\(\gcd(B, \widetilde{C}_{j}) = \gcd(B, c_{j}^{2}) = 1\) for \(j = 1, 2\).

Next, we choose auxiliary primes \(r_{i, j}\).
Fix \(i\) and \(j\), and write \(p = p_{i}\) and \(K = K_{i}\).
Since \(\rho_{j} \nmid 2c_{j}\), we have
\[
v_{\rho_{j}}(\widetilde{A}_{j}) = -1.
\]
Hence \(\widetilde{A}_{j} \notin \QQ^{\times p}\), and
\(\widetilde{A}_{j} \notin K^{\times p}\) by \([K:\QQ] = p - 1\) and Lemma~\ref{lem:coprime-degree-nth-powers}.
To clear the denominator, put
\[
\gamma_{j} = \widetilde{A}_{j}\widetilde{C}_{j}^{p}
= 4c_{j}^{4}\widetilde{C}_{j}^{p - 1} \in \ZZ \setminus \{0\}, \qquad
L = K(\sqrt[p]{\gamma_{j}}) = K(\sqrt[p]{\widetilde{A}_{j}}).
\]
Since \(\widetilde{A}_{j} \notin K^{\times p}\), Proposition~\ref{prop:kummer-basic}~(i) and (iii) show that \(L/K\) is cyclic of degree \(p\).
Since \(L\) is the splitting field of \(X^{p} - \gamma_{j}\) over \(\QQ\), the extension \(L/\QQ\) is Galois.

Take a nontrivial element \(\tau\) of \(\Gal(L/K)\),
and let \(\mathcal{C}\) be its conjugacy class in \(\Gal(L/\QQ)\).
Since \(K/\QQ\) is a Galois extension,
every element of \(\mathcal{C}\) restricts to the identity on \(K\) and is different from \(\id_{L}\).
By Chebotarev's density theorem (Corollary~\ref{cor:chebotarev-rational}), there exist infinitely many rational primes \(r\) unramified in \(L/\QQ\) whose Frobenius class is \(\mathcal{C}\).
Choose one that does not divide \(2nc_{1}c_{2}\widetilde{C}_{1}\widetilde{C}_{2}\).

Such an \(r\) splits completely in \(K/\QQ\), and thus
\(r \equiv 1 \pmod{p}\).
Let \(\mathfrak{R}\) be a prime ideal of \(L\) lying above \(r\),
and let \(\mathfrak{r}\) be the prime ideal of \(K\) below it. Since
\(f(\mathfrak{r}/(r)) = 1\),
\[
\Frob{L/K}{\mathfrak{R}}
= \Frob{L/\QQ}{\mathfrak{R}} \ne \id_{L}.
\]
Applying Proposition~\ref{prop:power-residue-criterion} with
\(F = \QQ\), \(\mathfrak{q} = (r)\), and \(\gamma = \gamma_{j}\),
we obtain \(\gamma_{j} \notin \FF_{r}^{\times p}\).
Since \(r \nmid \widetilde{C}_{j}\) and
\(\gamma_{j}/\widetilde{A}_{j} = \widetilde{C}_{j}^{p}\),
we also have \(\widetilde{A}_{j} \notin \FF_{r}^{\times p}\).

Since there are infinitely many such primes for each \((i, j)\), we may choose the \(r_{i, j}\) successively, excluding the finitely many primes already chosen, so that they are mutually distinct and satisfy
\[
\begin{gathered}
r_{i, j} \equiv 1 \pmod{p_{i}}, \qquad
\widetilde{A}_{j} \notin \FF_{r_{i, j}}^{\times p_{i}}, \\
r_{i, j} \nmid 2nc_{1}c_{2}\widetilde{C}_{1}\widetilde{C}_{2}
\qquad (1 \le i \le \omega(n),\ j = 1, 2).
\end{gathered}
\]
In what follows, we fix all these primes.
Here the non-residuosity condition for each \(\widetilde{A}_{j}\) is imposed at a separate prime, so that no linear independence between the classes of \(\widetilde{A}_{1}\) and \(\widetilde{A}_{2}\) is needed.

Next, we construct imaginary quadratic fields by imposing congruence conditions.
Let \(\widetilde{\mathcal{T}}\) be the set of all integers \(t \ge 2B\) satisfying the following congruence conditions:
\[
\begin{aligned}
t &\equiv -B \pmod{\widetilde{M}}, \\
t &\equiv c_{1} \pmod{8}, \\
t &\equiv 0 \pmod{r_{i, j}}
\qquad (1 \le i \le \omega(n),\ j = 1, 2).
\end{aligned}
\]
Since \(\widetilde{M}\) is odd and the moduli are pairwise coprime,
the set \(\widetilde{\mathcal{T}}\) is infinite by the Chinese remainder theorem.
Note that all the auxiliary primes are fixed before \(t\) is chosen.

Let \(t \in \widetilde{\mathcal{T}}\).
Since \(n\) is odd and \(t \equiv -B \pmod{\widetilde{C}_{j}}\), the integer \(\widetilde{C}_{j}\) divides \(t^{n} + B^{n}\). Hence, if we put
\[
\widetilde{k}_{j}(t) = \frac{t^{n} + B^{n}}{\widetilde{C}_{j}},
\]
then \(\widetilde{k}_{j}(t)\) is an integer, and \(\widetilde{k}_{j}(t) > 0\) since \(t \ge 2B > 0\). We also put
\[
\begin{aligned}
\widetilde{\Delta}_{j}(t)
&= -t^{2n} - 2B^{n}t^{n} + c_{j}^{2}
= \widetilde{C}_{j} - (t^{n} + B^{n})^{2}, \\
\widetilde{D}_{j}(t) &= \frac{\widetilde{\Delta}_{j}(t)}{16}, \qquad
\widetilde{F}_{j}(t) = \QQ\left(\sqrt{\widetilde{D}_{j}(t)}\right)
= \QQ\left(\sqrt{\widetilde{\Delta}_{j}(t)}\right).
\end{aligned}
\]
Since \(B\) is even and \(n \ge 3\), we have \(16 \mid 2B^{n}\).
In addition, \(t \equiv c_{1} \pmod{8}\) gives \(t^{2} \equiv c_{1}^{2} \pmod{16}\), and
since \(c_{1}\) is odd and \(n\) is odd, this yields
\[
t^{2n} \equiv c_{1}^{2n} \equiv c_{1}^{2}
\equiv c_{j}^{2} \pmod{16}.
\]
Therefore \(\widetilde{D}_{j}(t) \in \ZZ\).
Since \(t \ge 2B\) and \(B > c_{1}\), we have \(t^{2n} > c_{j}^{2}\), so that
\(\widetilde{\Delta}_{j}(t) < 0\).
By~\eqref{eq:c-square-difference},
\[
\widetilde{D}_{1}(t) = \widetilde{D}_{2}(t) + m < 0.
\]

Furthermore, since
\[
\widetilde{\Delta}_{j}(t)
= \widetilde{C}_{j}\bigl(1 - \widetilde{C}_{j}\widetilde{k}_{j}(t)^{2}\bigr), \qquad
1 - \widetilde{C}_{j}\widetilde{k}_{j}(t)^{2} \equiv 1 \pmod{\rho_{j}},
\]
it follows that
\[
v_{\rho_{j}}\bigl(\widetilde{\Delta}_{j}(t)\bigr)
= v_{\rho_{j}}\bigl(\widetilde{D}_{j}(t)\bigr) = 1.
\]
Hence the odd prime \(\rho_{j} \ge 5\) ramifies in \(\widetilde{F}_{j}(t)/\QQ\).
Since the discriminants of \(\QQ(\sqrt{-1})\) and \(\QQ(\sqrt{-3})\) are \(-4\) and \(-3\), respectively,
the field \(\widetilde{F}_{j}(t)\) is distinct from both exceptional fields.

Next, we define the integers needed to apply Lemma~\ref{lem:imaginary-class-group-criterion}.
For each \(j \in \{1, 2\}\), put
\[
\widetilde{S}_{j}(t)
= \frac{4(c_{j}^{2} - B^{n}t^{n})^{2}}{\widetilde{C}_{j}}.
\]
From \(t^{n} = \widetilde{C}_{j}\widetilde{k}_{j}(t) - B^{n}\), we deduce
\[
c_{j}^{2} - B^{n}t^{n}
= \widetilde{C}_{j}\bigl(1 - B^{n}\widetilde{k}_{j}(t)\bigr).
\]
Furthermore, since \(t \ge 2B\) and \(B > c_{1}\), we see that \(c_{j}^{2} - B^{n}t^{n} < 0\). Consequently,
\[
\widetilde{S}_{j}(t)
= 4\widetilde{C}_{j}\bigl(1 - B^{n}\widetilde{k}_{j}(t)\bigr)^{2},
\]
which is a positive integer.

From \(t \equiv -B \pmod{\widetilde{M}}\) and \(\gcd(B, \widetilde{M}) = 1\), we obtain
\(\gcd(t, \widetilde{C}_{j}c_{j}) = 1\), and
\(t\) is odd since \(t \equiv c_{1} \pmod{8}\).
For any prime \(q\) dividing \(t\), we have \(q \ne 2\) and \(q \nmid c_{j}\), and hence
\[
\widetilde{C}_{j}\widetilde{S}_{j}(t)
= 4(c_{j}^{2} - B^{n}t^{n})^{2}
\equiv 4c_{j}^{4} \not\equiv 0 \pmod{q}.
\]
Thus \(q \nmid \widetilde{S}_{j}(t)\), and
\[
\gcd\bigl(\widetilde{S}_{j}(t), t^{2}\bigr) = 1.
\]
In addition, \(t \equiv 0 \pmod{r_{i, j}}\) and \(r_{i, j} \nmid \widetilde{C}_{j}\) give
\[
\widetilde{S}_{j}(t)
\equiv \frac{4c_{j}^{4}}{\widetilde{C}_{j}}
= \widetilde{A}_{j} \pmod{r_{i, j}}.
\]
By the choice of the auxiliary primes, \(\widetilde{S}_{j}(t)\) is therefore a \(p_{i}\)-th power non-residue modulo \(r_{i, j}\).

A direct computation gives the following relation between the radicands:
\[
(c_{j}^{2} - B^{n}t^{n})^{2} - \widetilde{C}_{j}t^{2n}
= c_{j}^{2}\widetilde{\Delta}_{j}(t).
\]
Hence, if we put
\[
\widetilde{\Delta}_{j}^{\ast}(t) = \widetilde{S}_{j}(t)^{2} - 4\widetilde{S}_{j}(t)t^{2n},
\]
then
\[
\begin{aligned}
\widetilde{\Delta}_{j}^{\ast}(t)
&= \left(\frac{4c_{j}(c_{j}^{2} - B^{n}t^{n})}{\widetilde{C}_{j}}\right)^{2}
\widetilde{\Delta}_{j}(t) \\
&= \bigl(4c_{j}(1 - B^{n}\widetilde{k}_{j}(t))\bigr)^{2}
\widetilde{\Delta}_{j}(t).
\end{aligned}
\]
Here \(1 - B^{n}\widetilde{k}_{j}(t) \ne 0\), and thus
\(\widetilde{\Delta}_{j}^{\ast}(t)\) is a negative integer. Furthermore,
\[
\QQ\left(\sqrt{\widetilde{\Delta}_{j}^{\ast}(t)}\right)
= \widetilde{F}_{j}(t).
\]
We apply Lemma~\ref{lem:imaginary-class-group-criterion} with
\[
(X, Z, D) = \bigl(\widetilde{S}_{j}(t), t^{2}, \widetilde{\Delta}_{j}^{\ast}(t)\bigr),
\qquad l_{i} = r_{i, j}.
\]
The coprimality and non-residuosity established above, together with \(r_{i, j} \equiv 1 \pmod{p_{i}}\) and \(r_{i, j} \mid t^{2}\), show that all the hypotheses of that lemma are satisfied.
Therefore the class group of \(\widetilde{F}_{j}(t)\) contains an element of order \(n\).
Applying this argument to \(j = 1\) and \(j = 2\) proves the claim for both fields.

Finally, we prove that infinitely many distinct pairs arise.
For each \(j\), put
\[
\widetilde{f}_{j}(X) = -X^{2n} - 2B^{n}X^{n} + c_{j}^{2}.
\]
This is a polynomial of degree \(2n \ge 6\), and
\[
\widetilde{f}_{j}'(X) = -2nX^{n - 1}(X^{n} + B^{n}).
\]
Any common root of \(\widetilde{f}_{j}\) and \(\widetilde{f}_{j}'\) must satisfy \(X = 0\) or \(X^{n} = -B^{n}\). At such points, \(\widetilde{f}_{j}\) takes the values \(c_{j}^{2}\) and \(\widetilde{C}_{j}\), respectively, both of which are nonzero. Hence \(\widetilde{f}_{j}\) has no multiple root.
Applying Lemma~\ref{lem:siegel-quadratic-fields} to
\(\widetilde{f}_{2}\) and the infinite set \(\widetilde{\mathcal{T}}\) shows that
\[
\left\{\widetilde{F}_{2}(t) : t \in \widetilde{\mathcal{T}}\right\}
= \left\{\QQ\left(\sqrt{\widetilde{f}_{2}(t)}\right) : t \in \widetilde{\mathcal{T}}\right\}
\]
is infinite. Hence infinitely many distinct pairs
\(\bigl(\widetilde{F}_{2}(t), \widetilde{F}_{1}(t)\bigr)\) are obtained.

Furthermore, \(\rho_{1}\) ramifies in \(\widetilde{F}_{1}(t)\), whereas
\[
\widetilde{\Delta}_{2}(t)
= \widetilde{\Delta}_{1}(t) - 16m
\equiv -16m \not\equiv 0 \pmod{\rho_{1}},
\]
so that \(\rho_{1}\) does not ramify in \(\widetilde{F}_{2}(t)\).
Therefore the two fields in each pair are distinct.
Consequently, if we put \(D = \widetilde{D}_{2}(t)\), then we obtain the desired infinite family of pairs of imaginary quadratic fields.
\end{proof}

\appendix
\section{The case \(l = 2\) of Conjecture~\ref{conj:iizuka}}\label{app:case-l-two}

As the author noted in \cite[\S3]{Iizuka2018},
the case \(l = 2\) of Conjecture~\ref{conj:iizuka} follows from genus theory and the Chinese remainder theorem.
In this appendix, we give a proof that makes explicit the distinction between the class number and the narrow class number for real quadratic fields,
as well as the existence of infinitely many distinct tuples of fields.

\begin{proposition}\label{prop:conjecture-l-two}
Conjecture~\ref{conj:iizuka} holds for \(l = 2\).
Moreover, the quadratic fields within each \((m + 1)\)-tuple can be chosen to be pairwise distinct.
\end{proposition}

\begin{proof}
By Lemma~\ref{lem:genus-two-part}, if \(K\) is a real quadratic field with \(\omega(d_{K}) \ge 3\), then \(2 \mid h(K)\); the same sufficient condition also holds when \(K\) is an imaginary quadratic field.

Fix an integer \(m \ge 1\), and let \(\mathcal{F}\) be an arbitrary finite set of quadratic fields.
The set \(\mathcal{F}\) may be empty.
Let \(\mathcal{S}\) be the set of all rational primes that ramify in at least one of the fields in \(\mathcal{F}\).
Since \(\mathcal{S}\) is a finite set,
we can choose \(3(m + 1)\) mutually distinct odd primes
\[
r_{j, k} > m, \qquad r_{j, k} \notin \mathcal{S}
\qquad (0 \le j \le m,\ 1 \le k \le 3).
\]
By the Chinese remainder theorem, the solutions of the system of congruence conditions
\[
D \equiv r_{j, k} - j \pmod{r_{j, k}^{2}}
\qquad (0 \le j \le m,\ 1 \le k \le 3)
\]
form a single residue class modulo \(\prod_{j = 0}^{m}\prod_{k = 1}^{3}r_{j, k}^{2}\).
Thus, in the real quadratic case we can choose a representative with \(D > 0\), whereas in the imaginary quadratic case we can choose one with \(D + m < 0\).

Then, for each \(j\) and \(k\), we have
\[
v_{r_{j, k}}(D + j) = 1.
\]
In particular, \(D + j\) is nonzero and not a square, and thus \(L_{j} = \QQ(\sqrt{D + j})\) is a quadratic field.
Let \(d_{j}\) be the square-free part of \(D + j\).
By the above valuation condition, the primes \(r_{j, 1}, r_{j, 2}, r_{j, 3}\) all divide \(d_{j}\).
Since the discriminant of \(L_{j}\) is \(d_{j}\) or \(4d_{j}\), these three odd primes ramify in \(L_{j}\),
and \(\omega(d_{L_{j}}) \ge 3\). It follows that all the class numbers \(h(L_{j})\) are even.

Furthermore, if \(j \ne j'\), then
\[
D + j' \equiv j' - j \not\equiv 0 \pmod{r_{j, 1}},
\]
where the last incongruence follows from \(0 < |j' - j| \le m < r_{j, 1}\).
Hence \(r_{j, 1}\) ramifies in \(L_{j}\) but not in \(L_{j'}\), and therefore \(L_{j} \ne L_{j'}\).
In addition, since \(r_{j, 1} \notin \mathcal{S}\), the field \(L_{j}\) does not belong to \(\mathcal{F}\).

Thus, given any finite set \(\mathcal{F}\) of quadratic fields, we can construct an \((m + 1)\)-tuple with the required properties, all of whose fields lie outside \(\mathcal{F}\).
Repeating this construction after adding all previously obtained fields to \(\mathcal{F}\) yields infinitely many distinct \((m + 1)\)-tuples.
At each stage, we choose the representative with \(D > 0\) in the real quadratic case and with \(D + m < 0\) in the imaginary quadratic case; this proves the assertion in both cases.
\end{proof}

\section*{Acknowledgements}

I would like to express my gratitude to Professor Shin Nakano and Professor Yutaka Konomi for their helpful comments and suggestions.

\end{document}